\documentclass[11pt,english,oneside,a4paper]{amsart}

\usepackage[english]{babel}
\usepackage[utf8]{inputenc}
\usepackage[T1]{fontenc}
\usepackage[bottom=3cm,a4paper]{geometry}
\usepackage[hidelinks]{hyperref}
\usepackage{amsmath,amsfonts,amssymb,mathtools,amsthm}
\usepackage{listings,verbatim,fancyvrb,xspace,enumerate}
\usepackage{xcolor,graphicx,ytableau,tikz}
\usepackage[all, cmtip]{xy}
\usepackage{fancyhdr}
\usepackage[toc,page]{appendix}
\usepackage{pdfpages}

\newtheorem{theorem}{Theorem}[subsection]
\newtheorem*{theorem*}{Theorem}

\newtheorem{corollary}[theorem]{Corollary}
\newtheorem*{corollary*}{Corollary}
\newtheorem{lemma}[theorem]{Lemma}
\newtheorem{proposition}[theorem]{Proposition}
\theoremstyle{definition}
\newtheorem{example}[theorem]{Example}
\newtheorem*{example*}{Example}
\newtheorem*{examples*}{Examples}
\newtheorem{remark}[theorem]{Remark}

\newtheorem{conjecture}{Conjecture}

\newtheorem*{ack}{Acknowledgements}

\newsavebox{\eqbox}
\newenvironment{longequation*} {\begin{lrbox}{\eqbox}$} {$\end{lrbox}\begin{equation*}\resizebox{\linewidth}{!}{\ensuremath{\displaystyle\usebox{\eqbox}}}\end{equation*}}

\makeatletter
\newcommand*\bigcdot{\mathpalette\bigcdot@{.5}}
\newcommand*\bigcdot@[2]{\mathbin{\vcenter{\hbox{\scalebox{#2}{\(\m@th#1\bullet\)}}}}}
\newcommand{\tikznode}[3][inner sep=0pt]{\tikz[remember picture,baseline=(#2.base)]{\node(#2)[#1]{\(#3\)};}}
\makeatother

\def\bydef{\coloneqq}

\DeclarePairedDelimiter{\abs}{\lvert}{\rvert}
\DeclarePairedDelimiter{\Set}{\lbrace}{\rbrace}

\DeclareMathOperator{\Span}{Span}
\DeclareMathOperator{\diff}{d}

\DeclareMathOperator{\Vect}{Vect}

\DeclareMathOperator{\length}{length}
\DeclareMathOperator{\Image}{Im}
\DeclareMathOperator{\Hom}{Hom}

\DeclareMathOperator{\Id}{Id}

\DeclareMathOperator{\Proj}{Proj}

\DeclareMathOperator{\Diag}{Diag}

\DeclareMathOperator{\Ker}{Ker}

\DeclareMathOperator{\GL}{GL}
\DeclareMathOperator{\Diff}{diff}
\DeclareMathOperator{\hw}{hw}
\DeclareMathOperator{\Wronsk}{Wronsk}
\DeclareMathOperator{\Trace}{Trace}
\DeclareMathOperator{\Spec}{Spec}
\DeclareMathOperator{\can}{can}
\DeclareMathOperator{\End}{End}

\newcommand{\C}{\mathbb{C}}
\renewcommand{\O}{\mathcal{O}} % ( was Ø )
\renewcommand{\P}{\mathbf{P}} % ( was ¶ )

\newcommand{\N}{\mathbb{N}}
\newcommand{\Z}{\mathbb{Z}}

\let\mathbi\boldsymbol
\author{Antoine Etesse}
\email{antoine.etesse@math.univ-toulouse.fr}
\address{Institut Mathématique de Toulouse (IMT), Université Paul Sabatier}
\title{On the Schmidt-Kolchin conjecture on differentially homogeneous polynomials. Applications to (twisted) jet differentials on projective spaces.}

\begin{document}
\sloppy

\begin{abstract}
The main goal of this paper is to prove the Schmidt--Kolchin conjecture. This conjecture says the following: the vector space of degree \(d\) differentially homogeneous polynomials in \((N+1)\) variables is of dimension \((N+1)^{d}\). 
Next, we establish a one-to-one correspondance between differentially homogeneous polynomials in \((N+1)\) variables, and twisted jet differentials on projective spaces. As a by-product of our study of differentially homogeneous polynomials, we are thus able to understand explicitly twisted jet differentials on projective spaces.
\end{abstract}
\maketitle
\tableofcontents

\section*{Introduction.}
The goal of this paper is to prove the Schmidt-Kolchin conjecture on differentially homogeneous polynomials, and give applications to Green--Griffiths bundles of projective spaces. In order to state this conjecture, let us first set some notations (which are going to be used throughout the whole paper). We fix a natural number \(N \in \N_{\geq 1}\), and consider the complex vector space of \textsl{differential polynomials}
\[
V\bydef \C\big[(X^{(k)})_{k \in N}]
\]
in the formal variables \(X^{(k)}\bydef (X_{0}^{(k)}, \dotsc, X_{N}^{(k)})\). An element in this vector space is usually called a \textsl{differential polynomial} (in the variables \(X=(X_{0}, \dotsc, X_{N})\)). The upper index in parenthesis is reminiscent of a formal derivation, for the following reason. Consider a formal variable \(T\), and for any differential polynomial \(P \in V\) and any polynomial \(Q \in \C[T]\), form a new polynomial \(Q \cdot P \in V[T]\) by setting
\[
Q\cdot P
\bydef 
P\big((QX)^{(0)}, (QX)^{(1)}, \dotsc\big).
\]
Here, for any \(k \in \N\), the symbol \((QX)^{(k)}\) is by definition:
\[
(QX)^{(k)}
\bydef
\sum\limits_{i=0}^{k}
\binom{k}{i}
Q^{(k-i)}X^{(i)}.
\]
Namely, one applies formally the usual Leibnitz rule.
There are two natural notions of homogeneity on \(V\). The first one is the usual one: a differential polynomial \(P \in V\) is \textsl{homogeneous of degree \(d\)} if and only if  the following equality holds for any \(\lambda \in \C\):
\[
\lambda \cdot P
=
\lambda^{d}
P.
\]
The second notion of homogeneity, called \textsl{differential homogeneity}, is defined as follows. A differential polynomial \(P \in V\) is called \textsl{differentially homogeneous of degree \(d\)} if and only if, for any polynomial \(Q \in \C[T]\), the following equality holds:
\[
Q \cdot P
=
Q^{d}
P.
\]
Accordingly, a differentially homogeneous polynomial of degree \(d\) is necessarily homogeneous of degree \(d\). Denote by 
\[
V^{\Diff} \subset V
\]
the sub-vector space of differentially homogeneous polynomials.

The vector space \(V\) is naturally filtered by the maximal order of derivation
\[
(0) \subset V^{(0)} \subset V^{(1)} \subset \dotsb \subset V^{(k)} \subset \dotsb,
\]
where for any \(k \in \N\), one defines \(V^{(k)}\bydef \C\big[(X^{(i)})_{0 \leq i \leq k}]\). Furthermore, each piece in the filtration admits a natural and compatible graduation given by the degree of homogeneity
\[
V^{(k)}
=
\bigoplus_{d \in \N}
V^{(k)}_{d},
\]
where \(V^{(k)}_{d}\) is the set of differential polynomials in \(V^{(k)}\) that are homogeneous of degree \(d\).
Both the filtration and the graduation descend to the sub-vector space \(V^{\Diff}\), and we will denote by
\[
(V^{\Diff}_{d})^{(k)}
\]
the sub-vector space of differentially homogeneous polynomials in \(V^{(k)}_{d}\).

We are now ready to state the so-called \textsl{Schmidt--Kolchin conjecture}:
\begin{conjecture}[Schmidt--Kolchin]
The dimension of the vector space \(V^{\Diff}_{d}\) of differentially homogeneous polynomials of degree \(d\) in the variables \(X=(X_{0}, \dotsc, X_{N})\) is equal to \((N+1)^{d}\).
\end{conjecture}
The origin of this questions goes back to Schmidt's paper \cite{Schmidt_1979}, where  he proved that the dimension is at least equal to \((N+1)^{d}\), and suggested that "Perhaps, equality is true here; no upper bound seems to be known at present". A few years later, Kolchin proved in \cite{Kolchin} that the conjecture holds for \(N=1\) and \(d=1,2,3\). Then, in \cite{Rein-Sit}, Reinhart and Sit showed that the following equality holds for \(N=1\):
\[
V^{\Diff}_{d}
=
(V^{\Diff}_{d})^{(d-1)}.
\]
In particular, this implies that the vector space of differentially homogeneous polynomials of a given degree \(d \in \N_{\geq 1}\) in the variables \(X=(X_{0}, X_{1})\) is finite dimensional. Three years later, Reinhart improved the previous result, and solved the Schmidt--Kolchin conjecture for \(N=1\) in \cite{Reinhart_1999}. Since then, no progress has been made towards this conjecture. It is not even known if the vector space \(V_{d}^{\Diff}\) is finite dimensional when \(N>1\).

In this paper, we solve the Schmidt--Kolchin conjecture, and prove the following (see also Theorem \ref{thm: mainthm1 corpus}):
\begin{theorem}
\label{thm: mainthm1}
The dimension of \(V^{\Diff}_{d}\) is equal to \((N+1)^{d}\). Furthermore, an explicit basis of this vector space is provided, and one deduces accordingly the equality
\[
V^{\Diff}_{d}
=
(V^{\Diff}_{d})^{(d-1)}.
\]
\end{theorem}
In a few words, the strategy of the proof is as follows (for more details, see Organization of the paper). One easily sees (see the beginning of Section \ref{sect: Schmidt-Kolchin}) that there is a natural action of the general linear group on \(V_{d}\), which descends to \(V_{d}^{\Diff}\). Fixing \(k \in \N\) a natural number, this makes \((V_{d}^{\Diff})^{(k)}\) into a sub-representation of \(V_{d}^{(k)}\). There is a natural basis of the vector space \(V_{d, \hw}^{(k)}\) of \textsl{highest weight vectors} of \(V_{d}^{(k)}\) (see Section \ref{subs: linear} for definitions, and Section \ref{subs: hw vectors}). Therefore, the problem amounts to understanding which suitable linear combinations of elements in this basis give differentially homogeneous polynomials. The key part in the proof lies in a suitable algebraic formulation of this problem, and the way it is tackled (see Section \ref{subs: pure algebraic} and Section \ref{subs: relate}).

Beside the intrinsic interest one may have in such a question, our motivation was rather geometric. It turns out that there is one-to-one correspondance between differentially homogeneous polynomial in the variables \(X=(X_{0}, \dotsc, X_{N})\) and \textsl{(positively) twisted jet differentials} on the projective space \(\P^{N}\) (see Section \ref{subs: GG} for definition of jet differentials on complex manifolds, and see Proposition \ref{prop: link}).
%The idea behind this correspondance is very simple and natural. Informally, the set of homogeneous polynomials consists of the algebraic "functions"\footnote{Global sections of natural line bundles to be more precise.} on the set of points on the projective space \(\P^{N}\). As a natural generalization, for any \(k \geq 1\), the set \((V^{\Diff})^{(k)}\) of differentially homogeneous polynomials of order at most \(k\)  consists of the algebraic "functions" on the set of points (order \(0\)), tangent vectors (order \(1\)), and so forth up to the order \(k\). This vague idea is made precise in Section \ref{subs: GG} via the introduction of the so-called \textsl{Green--Griffiths bundles} (whose global sections are called \textsl{jet differentials}).
As a corollary of our study on differentially homogeneous polynomials, we obtain the second main theorem of this paper (see Section \ref{subs: GG} for notations and definitions):
\begin{theorem}
\label{thm: mainthm2}
Let \(d \in \Z\), \(k \geq 0\) and \(n \geq 0\) be natural numbers. Then:
\begin{enumerate}
\item{} for any \(k \geq d-1\) and any \(n \in \N\), one has the isomorphism
 \[
 H^{0}(\P^{N}, E_{k,n}\P^{N}(d)) \simeq H^{0}(\P^{N}, E_{d-1,n}\P^{N}(d));
 \]
 \item{} the following equality holds:
\[
\sum\limits_{n=0}^{\infty} \dim H^{0}(\P^{N},E_{d-1,n}\P^{N}(d)) = (N+1)^{d};
\]
\item{} for any \(k \in \N\) and any \(n > \lfloor(1-\frac{1}{N+1})\frac{d^{2}}{2}\rfloor\), one has the vanishing: 
\[
H^{0}(\P^{N}, E_{k,n}\P^{N}(d))=(0).
\]
\end{enumerate}
\end{theorem}
It should be noted that the previous Theorem \ref{thm: mainthm2} does not encapsulate all the information we have on twisted jet differentials on projective spaces. Indeed, since differentially homogeneous polynomials are explicitly understood, so are these twisted jet differentials (see Proposition \ref{prop: link}). Whereas, being given a differentially homogeneous in the \textsl{canonical basis} of \(V^{\Diff}\) (see Section \ref{subs: construction}), it is easy to determine which global section it defines (see Proposition \ref{prop: isobar}), it seems to be quite difficult to count the number of differentially homogeneous polynomials falling into a same class. Such a purely combinatoric problem might be the object of a future work, allowing to obtain closed formulas for the dimension of the space of global sections of Green--Griffiths bundles of projective spaces.

To finish this introductory part, we would like to mention another motivation. As any algebraic geometer knows, to any graded ring \(S\) is associated the scheme \(\Proj(S)\), which is of particular interest if \(S\) is furthermore finitely generated. It is therefore very natural to study, for any \(k \in \N\), the (bigger and bigger as \(k\) grows) scheme 
\[
\Proj\big(\bigoplus_{d \geq 0} (V_{d}^{\Diff})^{(k)}\big),
\]
along with its closed subschemes. This is the object of an ongoing work.
\subsection*{Organization of the paper.}
The paper is organized as follows.

Section \ref{sect: preliminaries} is devoted to preliminaries.

In Section \ref{subs: representation}, we recall some classic facts on representation theory of symmetric groups and general linear groups, allowing along the way to fix notations. 

In Section \ref{subs: PDE}, we recall a correspondance between systems of zero-dimensional partial differential equations with constant coefficients on the one hand, and algebraic varieties on the other hand. Such a correspondance allows to evaluate the number of independent solutions of a system of PDE's via methods coming from intersection theory.

Section \ref{sect: Schmidt-Kolchin} is devoted to the proof of the Schmidt--Kolchin conjecture.

In Section \ref{subs: family}, following \cite{Reinhart_1999}, we introduce, for each integer \(d \geq 0\), a natural sub-family \(\overline{V^{\Diff}_{d}} \subset V^{\Diff}_{d}\) of differentially homogeneous polynomials of degree \(d\), which is shown to be \((N+1)^{d}\)-dimensional.\footnote{As indicated in the Introduction, the lower bound
\[
\dim(V^{\Diff}_{d}) \geq (N+1)^{d}
\]
was already proved by Schmidt in \cite{Schmidt_1979}. However, he did not prove this result by exhibiting an explicit free family of differentially homogeneous polynomials.}
The main result of this Section \ref{subs: family} is the following: we show that \(\overline{V_{d}^{\Diff}}\) is actually a sub-representation of \(V_{d}^{\Diff}\) for the natural action of the general linear group.\footnote{Believing in the Schmidt--Kolchin conjecture, this is the least one could expect.} As a crucial corollary, one deduces a lower bound on the dimension of \(\mathbi{\lambda}\)-highest weight vectors in \(V^{\Diff}_{d}\), for \(\mathbi{\lambda} \vdash d\) a partition of \(d\) with at most \((N+1)\) parts:
\begin{equation}
\label{eq: eqintro1}
\dim V^{\Diff}_{\hw(\mathbi{\lambda})}
\geq
f_{\mathbi{\lambda}}.
\end{equation}

In Section \ref{subs: hw vectors}, we fix \(k \in \N_{\geq 0}\) an integer, and provide a natural basis for the vector space
\[
V_{d, \hw}^{(k)}
\]
of highest weight vectors in the finite-dimensional representation \(V_{d}^{(k)}\).
%Recall indeed that by the above Lemma \ref{lemma: elementary}, each piece \(V_{d}^{(k)}\) forms a finite-dimensional representation of \(\GL_{N+1}(\C)\) (which contains \((V^{\Diff}_{d})^{(k)}\) as a sub-representation). 

In Section \ref{subs: pure algebraic}, we study a purely algebraic problem, which consists in the evaluation of the dimension of the space of vectors lying simultaneously in the kernel of a family of particular endomorphisms of a fixed finite-dimensional vector space. We feel that one of the main originality of this paper lies in the way we tackled this problem, as we crucially used the correspondance between systems of PDE's and algebraic varieties, briefly described above.

In Section \ref{subs: relate}, we provide a set of equations characterizing highest weight vectors of \(V_{d}^{(k)}\) that are differentially homogeneous. The technical part of this Section \ref{subs: relate} amounts to relate this set of equations to the previous algebraic problem. At this stage, the small miracle is that the dependency on \(N\) and \(k\) (almost) vanishes. The outcome of this Section \ref{subs: relate} is that we have an upper bound for the number of independent \(\mathbi{\lambda}\)-highest weight vector in \((V_{d}^{\Diff})^{(k)}\) (where \(\mathbi{\lambda}\vdash d\) is partition with at most \((N+1)\) parts), which depends only on \(d\).\footnote{Accordingly, at this stage, one has already proved that \(V_{d}^{\Diff}\) is finite dimensional.}

In Section \ref{subs: proof}, we finish the proof of the Schmidt--Kolchin conjecture, and to this end, the lower bound \eqref{eq: eqintro1} turns out to be crucial.

Section \ref{sect: diff jets} is devoted to establishing the correspondance between differentially homogeneous polynomials and twisted jet differentials on projective spaces, and infer applications from the previous study.

In Section \ref{subs: GG}, we recall the construction of Green--Griffiths bundles on complex manifolds.

In Section \ref{subs: link}, we establish the correspondance between differentially homogeneous polynomials, and twisted jet differentials on projective spaces.

In Section \ref{subs: H^0 GG}, we study the \(0\)th cohomology group (i.e. the vector space of global sections) of twisted Green--Griffiths bundles via the acquired knowledge on differentially homogeneous polynomials.

\begin{ack}
I would like thank Gleb Pogudin, as I learnt from him the terminology of differential homogeneity, as well as the Schmidt--Kolchin conjecture. It turns out that I had used the defining property of differential homogeneity in a previous paper, and called it \textsl{geometric}, out of ignorance. This is precisely the geometric implications of the Schmidt--Kolchin conjecture that motivated me to tackle it.

I also would like to thank the Institut de Mathématiques de Toulouse (IMT) for providing a very comfortable and stimulating working environment. And of course, I am very grateful to the CIMI LabEx for giving me the opportunity to work here.
\end{ack}

\section{Preliminaries.}
\label{sect: preliminaries}
In this Section \ref{sect: preliminaries}, we gather a few results that will be used during our study of differentially homogeneous polynomials. These results come from representation theory on the one hand (Section \ref{subs: representation}), and systems of partial differential equations on the other hand (Section \ref{subs: PDE}).

\subsection{Representation theory for symmetric groups and general linear groups.}
\label{subs: representation}
The references for this section are \cite{Fulton}, \cite{FultonHarris} and \cite{Boerner}.
\subsubsection{Partitions, Young tableaux, and some combinatorial identities.}
Recall that a \textsl{partition \(\mathbi{\lambda}\) of a natural number \(d\geq 1\)} is simply a finite, decreasing sequence of positive natural numbers whose sum equals to \(d\).
To any partition \(\mathbi{\lambda}\) with \(s\) parts
\[
\mathbi{\lambda}=(\lambda_{1}\geq\dotsb\geq\lambda_{s}>0)
\]
is associated a \textsl{Young diagram of shape \(\mathbi{\lambda}\)}, which is a collection of cells arranged in left-justified rows: the first row contains \(\lambda_{1}\) cells, the second \(\lambda_{2}\), and so on (cf. Figure~\ref{fig1} for an example).
\begin{figure}[h]
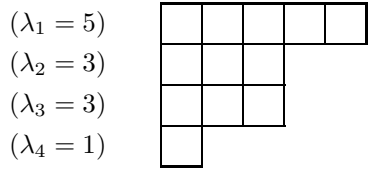

  \begin{center}
    \small
    \ytableausetup{centertableaux}
    \begin{ytableau}
      \none[(\lambda_{1}=5)]&\none[]&\none[]&&&&&&\none[\tikznode{s3}{~}]
      \\
      \none[(\lambda_{2}=3)]&\none[]&\none[]&&&
      \\
      \none[(\lambda_{3}=3)]&\none[]&\none[]&&&
      \\
      \none[(\lambda_{4}=1)]&\none[]&\none[]& & \none[]& \none[] & \none[]& \none[]& \none[\tikznode{s1}{~}]
    \end{ytableau}
    \caption{Young diagram of the partition \(\lambda=(5,3,3,1)\).}
    \label{fig1}
  \end{center}
\end{figure}
 %By construction, if we read the Young diagram from top to bottom via the rows, we recover the partition \(\lambda\). If instead we read it from left to right via the columns, we find the so-called \textsl{conjugate partition} \(\mathbi{\lambda}^{*}\).
%There is a \textsl{jump sequence} \(\mathbi{s}=(s_{1}> \dotsb > s_{t})\) associated to a partition \(\mathbi{\lambda}\), defined by:
%\[
%  \lambda_{i}>\lambda_{i+1}
%  \iff
%  (i\in \mathbi{s})
%\]
%(where by convention \(\lambda_{k+1}=0\)).
%Equivalently, \(s_{1}>\dotsb>s_{t}\) are the (pairwise distinct) lengths of the parts of the conjugate partition \(\mathbi{\lambda}^{\ast}\) (cf. Figure~\ref{fig1}).

Fix \(n \in \N\) a natural number, and \(\mathcal{T}\) a Young diagram of shape \(\mathbi{\lambda}\). A \textsl{Young tableau filled with the set \(\Set{1, \dotsc, n}\)} is, by definition, a filling \(T\) of the boxes of the Young diagram \(\mathcal{T}\) with the numbers \(\Set{1, \dotsc, n}\). Furthermore, if one imposes that 
\begin{itemize}
\item{} browsing from left to right a line of the diagram, the sequence of numbers is non-decreasing;
\item{} browsing from top to bottom a column of the diagram, the sequence of numbers is increasing.
\end{itemize}
then such a Young tableau is called a \textsl{semi-standard Young tableau of shape \(\mathbi{\lambda}\), filled with the set \(\Set{1, \dotsc, n}\)}. In the case where \(n=d\), and where one imposes furthermore that each number \(1 \leq k \leq d\) occurs exactly once in the tableau, then the tableau \(T\) is called a \textsl{standard tableau of shape \(\mathbi{\lambda}\)}.

One is naturally lead to consider the following (combinatorial) numbers:
\begin{itemize}
\item{}
 \(f_{\mathbi{\lambda}} \bydef \big|\big\{\text{ \(T\) standard Young tableau of shape \(\mathbi{\lambda}\)\big\}} \big|\);
\item{}
\(d_{\mathbi{\lambda}}(n) \bydef \big|\big\{\text{ \(T\) semi-standard Young tableau of shape \(\mathbi{\lambda}\), filled with \(\Set{1, \dotsc, n}\)\big\}} \big|\),
\end{itemize}
Furthermore, for a fixed \(n\)-uple \(\mathbi{a}=(a_{1}, \dotsc, a_{n}) \in \N^{n}\), with \(\abs{\mathbi{a}}=\abs{\mathbi{\lambda}}=d\), define the so-called \textsl{Kostka numbers}:
\[
K_{\mathbi{\lambda}, \mathbi{a}}
\bydef
 \big|\big\{\text{ \(T\) semi-standard Young tableau of shape \(\mathbi{\lambda}\), filled with \(a_{1}\) \(1\)'s ,\(\dotsc\), \(a_{n}\) \(n\)'s}\big\}\big|.
\]
Note that, from the very definition of Kostka numbers, the following equality holds:
\begin{equation}
\label{eq: Kotska}
d_{\mathbi{\lambda}}(n)
=
\sum\limits_{\abs{\mathbi{a}}=d}
K_{\mathbi{\lambda}, \mathbi{a}}.
\end{equation}

The \textsl{Robinson--Schensted--Knuth correspondance} (see e.g. \cite{Fulton}[I.4]) allows to obtain the following identities:
\begin{proposition}
\label{prop: RSK}
The following equalities hold for any \(d \geq 1\) and any \(n \geq 1\):
\begin{itemize}
\item{} \(d!=\sum\limits_{\mathbi{\lambda} \vdash d} f_{\mathbi{\lambda}}^{2} \);
\item{} \(n^{d}=\sum\limits_{\mathbi{\lambda} \vdash d} f_{\mathbi{\lambda}}d_{\mathbi{\lambda}}(n)\).
\end{itemize}
Here, the symbol \(\mathbi{\lambda} \vdash d\) means that \(\mathbi{\lambda}\) is a partition of the natural number \(d\). 
\end{proposition}

\subsubsection{Representations of symmetric groups.}
\label{subs: sym}
Let \(d\in \N_{\geq 1}\), and denote by \(\Sigma_{d}\) the symmetric group of the set \(\Set{1, 2, \dotsc, d}\). It is well-known (see e.g. \cite{FultonHarris}[Lectures 1 \& 2]) that there are as many irreducible representations of \(\Sigma_{d}\) as there are conjugacy classes in \(\Sigma_{d}\). Note that the conjugacy classes are in one-to-one correspondance with partitions of \(d\).

Denote by \(\C[\Sigma_{d}]\) the \textsl{regular represention} of \(\Sigma_{d}\), i.e.
\[
\C[\Sigma_{d}]
\bydef
\bigoplus_{\sigma \in \Sigma_{d}}
\C \cdot \sigma.
\]
Note that \(\C[\Sigma_{d}]\) is tautologically a representation of the symmetric group \(\Sigma_{d}\). Note also that any representation \((M, \rho: \Sigma_{d} \to \GL_{\C}(M))\) of the symmetric group \(\Sigma_{d}\) has a natural structure of (left) \(\C[\Sigma_{d}]\)-module defined as follows. For \(m \in M\) and \((\lambda_{\sigma} \in \C)_{\sigma \in \Sigma_{d}}\), set:
\[
(\sum\limits_{\sigma \in \Sigma_{d}} \lambda_{\sigma} \sigma)
\cdot
m
\bydef
\sum\limits_{\sigma \in \Sigma_{d}} \lambda_{\sigma} \rho(\sigma)(m).
\]

Every irreducible representations of \(\Sigma_{d}\) appears in \(\C[\Sigma_{d}]\) (this is actually a general fact for the regular representation of any finite group: see e.g. \cite{FultonHarris}[Lecture 2]). One has furthermore canonical projections onto irreducible factors, which are defined as follows. Let \(\mathbi{\lambda} \vdash d\) be a partition of \(d\), and let \(T\) be a standard Young tableau with shape \(\mathbi{\lambda}\). To the standard Young tableau \(T\) is associated the following two subgroups of \(\Sigma_{d}\):
\begin{itemize}
\item{} the subgroup \(R(T)\) of permutations that preserve the rows of \(T\);
\item{} the subgroup \(C(T)\) of permutations that preserves the columns of \(T\).
\end{itemize}
The so-called \textsl{Young symmetrizer associated to the tableau \(T\)} is the vector in \(\C[\Sigma_{d}]\) defined as follows:
\[
c_{T}
\bydef
\underbrace{\big(\sum\limits_{q \in C(T)} \epsilon(q)q\big)}_{\bydef b_{T}}
\times
\underbrace{\big(\sum\limits_{p \in R(T)} p\big)}_{\bydef a_{T}},
\]
where \(\epsilon(\cdot)\) is the signature.  One has the following theorem (see e.g. \cite{Boerner}[IV. 3]):
\begin{theorem}
\label{thm: struct sym}
The following holds.
\begin{itemize}
\item{}
The multiplication map
\[
p_{T}\colon
\left(
\begin{array}{ccc}
\C[\Sigma_{d}] & \longrightarrow  &  \C[\Sigma_{d}]
 \\
 v & \longmapsto  & v \times c_{T}
 \end{array}
\right)
\]
is an almost-projection\footnote{By definition, this means that \(p_{T} \circ p_{T}= \lambda p_{T}\) for some non-zero scalar \(\lambda\).} onto an irreducible representation of \(\Sigma_{d}\).
\item{}
For two different standard tableaux \(T\) and \(T'\) with same shape \(\mathbi{\lambda} \vdash d\), the irreducible representations \(\Image(p_{T})\) and \(\Image(p_{T'})\) are isomorphic. The abstract irreducible representation is usually denoted \(S^{\mathbi{\lambda}}\).
\item{} 
For two standard tableaux \(T\) and \(T'\) with different shape, the irreducible representations \(\Image(p_{T})\) and \(\Image(p_{T'})\) are non-isomorphic.
\item{}
The representation \(\C[\Sigma_{d}]\) splits as follows into direct sum of irreducible representations:
\[
\C[\Sigma_{d}]
=
\bigoplus_{T} \Image(p_{T}),
\]
where \(T\) runs over standard tableaux with \(d\) boxes.
\end{itemize}
\end{theorem}

\subsubsection{Representations of general linear groups.}
\label{subs: linear}
Let \(n \in \N_{\geq 1}\), and consider the general linear group \(\GL_{n}(\C)\), along with the subgroups \(D \subset B \subset \GL_{n}(\C)\) where
\begin{itemize}
\item{} \(D\) is the subgroup of diagonal matrixes;
\item{} \(B\) is the Borel subgroup of upper triangular matrixes.
\end{itemize}
Fix \(E\) a finite-dimensional polynomial\footnote{This means that the map \(\rho: \GL_{n}(\C) \to \GL(E)\) is polynomial, i.e. after choosing a basis for \(E\), the \(\dim(E)^{2}\) coordinates functions of \(\rho\) are polynomial in the \(n^{2}\) variables of \(\GL_{n}(\C)\).} representation of \(\GL_{n}(\C)\). A \textsl{weight vector with weight \(\mathbi{\lambda}=(\lambda_{1}, \dotsc, \lambda_{n}) \in \N^{n}\)} is a vector \(e \in E\) such that, for every \(x_{1}, \dotsc, x_{n} \in \C\), the following equality holds:
\[
\Diag(x_{1}, \dotsc, x_{n}) \cdot e
=
(x_{1}^{\lambda_{1}} \dotsb x_{n}^{\lambda_{n}}) e.
\]
A \textsl{\(\mathbi{\lambda}\)-highest weight vector} is a weight vector of weight \(\mathbi{\lambda}\) which is left invariant, up to homothety, by \(B\). Namely, the following equality holds:
\[
B \cdot e
=
\C^{*} e.
\]

General theory of representations (see e.g \cite{Fulton}[II.8] and \cite{FultonHarris}[Lectures 14 \& 15]) tells that finite-dimensional polynomial representations are uniquely determined by their highest weight vectors. More precisely, one has that 
\begin{itemize}
\item{} every representation splits as a direct sum of irreducible representations;
\item{} a representation is irreducible if and only if it has a unique highest weight vector, up to a multiplicative factor;
\item{} two irreducible representations are isomorphic if and only if their highest weight vector has same weight.
\end{itemize}
Furthermore, in the case of the general linear group, possible highest weight vectors \(\mathbi{\lambda} \in \N^{n}\) are in one-to-one correspondance with partitions with at most \(n\) parts. Said otherwise, to any \(\mathbi{\lambda} \in \N^{n}\) satisfying
\[
\lambda_{1} \geq \lambda_{2} \geq \dotsb \geq \lambda_{n} \geq 0
\]
is associated a unique irreducible representation of \(\GL_{n}(\C)\), and vice-versa. This unique representation is denoted \(S^{\mathbi{\lambda}}\C^{n}\), and called \(\mathbi{\lambda}\)th Schur power of \(\C^{n}\). For later use, let us record the following classic proposition (see e.g. \cite{Fulton}[II.8.3]):
\begin{proposition}
\label{prop: dim Schur}
The dimension of the representation \(S^{\mathbi{\lambda}}\C^{n}\) is equal to \(d_{\mathbi{\lambda}}(n)\).
\end{proposition}

To any finite-dimensional polynomial representation \(E\) is associated its \textsl{character}, denoted \(\chi_{E}\), which is by definition the symmetric polynomial in the variables \(\mathbi{x}=(x_{1}, \dotsc, x_{n})\)
\[
\chi_{E}(\mathbi{x})
\bydef
\Trace(\Diag(\mathbi{x})).
\]
In the case where \(E=S^{\mathbi{\lambda}}\C^{n}\), the character \(\chi_{E}\) is called the \textsl{\(\mathbi{\lambda}\)th Schur polynomial} (in the variables \(\mathbi{x}=(x_{1}, \dotsc, x_{n})\)), and is denoted by \(s_{\mathbi{\lambda}}\). General theory on symmetric polynomials (see e.g. \cite{Fulton}[I.2.2]) shows that Schur polynomials form a basis of symmetric polynomials. 
Accordingly, one deduces that a representation \(E\) of the general linear group is uniquely determined by its character \(\chi_{E}\).

There is a natural interaction between representations of symmetric groups and general linear groups given as follows. Let \(E=(\C^{n})^{\otimes d}\). On the one hand, it is a natural representation of the general linear group \(\GL_{n}(\C)\). On the other hand, it has a natural structure of \textsl{right} \(\C[\Sigma_{d}]\)-module given by permutations of factors of tensor products. Namely, for \(\sigma \in \Sigma_{d}\) and \(v_{1}, \dots, v_{d} \in \C^{n}\), one sets:
\[
(v_{1} \otimes \dotsb \otimes v_{d})\cdot \sigma
\bydef
v_{\sigma(1)}\otimes \dotsb \otimes v_{\sigma(d)}.
\]
The simple but important observation is that these actions are \textsl{compatible}, in the sense that they commute with each other:
\[
A \cdot ((v_{1} \otimes \dotsb \otimes v_{d}) \cdot \sigma)
=
(A \cdot (v_{1} \otimes \dotsb \otimes v_{d})) \cdot \sigma.
\]
Here, \(A\) is a matrix in \(\GL_{n}(\C)\), \(\sigma\) is an element in \(\Sigma_{d}\), and \(v_{1}, \dotsc, v_{d}\) are vectors in \(\C^{n}\).

For any representation \(M\) of \(\Sigma_{d}\), seen as a left \(\C[\Sigma_{d}]\)-module, denote by
\[
E(M)
\bydef
E
\otimes_{\C[\Sigma_{d}]}
M
\]
the tensor product over \(\C[\Sigma_{d}]\) between \(E\) and \(M\).
In particular, one has tautologically the equality
\[
E(\C[\Sigma_{d}])
=
E.
\]
Following what was seen in the previous Section \ref{subs: sym}, one is naturally lead to consider, for \(T\) a standard Young tableau with \(d\) boxes, the linear map
\[
E(p_{T})\colon
\left(
\begin{array}{ccc}
E & \longrightarrow  &  E
 \\
 v & \longmapsto  & v \otimes_{\C[\Sigma_{d}]} c_{T}
 \end{array}
\right).
\]
Using in particular Theorem \ref{thm: struct sym}, one can show the following (see e.g. \cite{Boerner}[V.4] and \cite{Fulton}[II.8]):
\begin{theorem}
\label{thm: struct gen}
The following holds.
\begin{itemize}
\item{}
The linear map
\[
E(p_{T})\colon
\left(
\begin{array}{ccc}
E & \longrightarrow  &  E
 \\
 v & \longmapsto  & v \otimes_{\C[\Sigma_{d}]} c_{T}
 \end{array}
\right)
\]
is an almost-projection onto an irreducible representation of \(\GL_{n}(\C)\).
\item{}
For two different standard tableaux \(T\) and \(T'\) with same shape \(\mathbi{\lambda}\vdash d\), the irreducible representations \(\Image(E(p_{T}))\) and \(\Image(E(p_{T'}))\) are isomorphic to the \(\mathbi{\lambda}\)th Schur power \(S^{\mathbi{\lambda}}\C^{n} \simeq E(S^{\mathbi{\lambda}})\).
\item{} 
For two standard tableaux \(T\) and \(T'\) with different shape, the irreducible representations \(\Image(E(p_{T}))\) and \(\Image(E(p_{T'}))\) are non-isomorphic.
\item{}
The representation \(E=(\C^{n})^{\otimes d}\) splits as follows into direct sum of irreducible representations:
\[
E
=
\bigoplus_{T} \Image(E(p_{T})),
\]
where \(T\) runs over standard tableaux with \(d\) boxes.
\end{itemize}

\end{theorem}

%For later use, let us record here the following classic proposition (see e.g. \cite{Fulton}[II.8.3]):
%\begin{proposition}
%\label{prop: char}
%Let \(E=(\C^{n})^{\otimes d}\) be the \(d\)th tensor product of \(\C^{n}\), with \(d \geq 1\). Then \(E\) splits as follows as a direct sum of irreducible representations of \(\GL_{n}(\C)\):
%\[
%E 
%\simeq
%\bigoplus_{\mathbi{\lambda} \vdash d}
%(S^{\mathbi{\lambda}}\C^{n})^{\oplus f_{\mathbi{\lambda}}}.
%\]
%In particular, one has the following identity:
%\[
%\chi_{E}(\mathbi{x})
%=
%(x_{1}+\dotsb + x_{n})^{d}
%=
%\sum\limits_{\mathbi{\lambda}\vdash d} f_{\mathbi{\lambda}}s_{\mathbi{\lambda}}.
%\]
%\end{proposition}

\subsection{Zero-dimensional systems of partial differential equations with constant coefficients.}
\label{subs: PDE}
Fix \(n \in \N_{\geq 1}\) a natural number, and denote by 
\(
\O(\C^{n})
\)
the set of entire functions on \(\C^{n}\) (i.e. the set of holomorphic maps in \(n\) variables defined everywhere). Consider 
\[
A \bydef \C[\partial_{1}, \dotsc, \partial_{n}]
\]
the polynomial ring in the formal variables \(\partial_{1}, \dotsc, \partial_{n}\). The set of holomorphic maps \(\O(\C^{n})\) has a natural structure of \(A\)-module defined as follows on monomials
\[
(\partial_{1}^{\alpha_{1}} \dotsb \partial_{n}^{\alpha_{n}}) \cdot P
\bydef
\frac{\partial^{\alpha_{1}+\dotsb+\alpha_{n}}P}{(\partial z_{1})^{\alpha_{1}} \dotsb (\partial z_{n})^{\alpha_{n}}},
\]
and extended by linearity.
To any ideal \(I \subset A\) is associated a system \(\mathcal{S}(I)\) of partial differential equations (PDE's):
\[
\mathcal{S}(I)
\bydef
\Set{P \in \O(\C^{n}) \ | \ E \cdot P \equiv 0 \ \forall E \in I}.
\]
Note that since \(A\) is Noetherian (by Hilbert's basis theorem), \(\mathcal{S}(I)\) is defined by a finite number of partial differential equations with constant coefficients.

Suppose that the ideal \(I\) is zero-dimensional, i.e. the affine variety \(\Spec(\frac{A}{I})\) is finite.
In this case, one can relate the dimension (as a \(\C\)-vector space) of the space of solutions \(\mathcal{S}(I)\) to the algebraic structure of the quotient space \(\frac{A}{I}\).
Indeed, one of the main result that initiated the theory of systems of PDE's with constant coefficients is the following theorem, proved in \cite{Palamodov}[VII] and \cite{Ehrenpreis}:
\begin{theorem}[]
The \(A\)-module \(\O_{\C^{n}}\) is injective, i.e. the functor
\[
\Hom_{A}(\cdot, \O_{\C^{n}})
\]
is exact.
\end{theorem}
Further properties of the \(A\)-module \(\O_{\C^{n}}\) were obtained by Oberst in \cite{Oberst90}, where he proved that \(\O_{\C^{n}}\) is a \textsl{large injective cogenerator}: see \textsl{loc.cit} for definitions and details. As a by-product, this allowed Oberst to prove the following result in \cite{Oberst96}, crucial to us:
\begin{theorem}
\label{thm: PDE}
Let \(I \subset A\) be zero-dimensional ideal. Then one has the following equality:
\[
\dim_{\C}\mathcal{S}(I)
=
\length(\frac{A}{I}).
\]
\end{theorem}
We refer for instance to \cite{Eisenbud}[2.4] for the notion of length of an Artinian module.

\section{Structure of differentially homogeneous polynomials.}
\label{sect: Schmidt-Kolchin}

There is a natural left action (by change of variables) of the general linear group \(\GL_{N+1}(\C)\) on the vector space \(V\) of differential polynomials in \((N+1)\) variables defined as follows. Consider \(A \in \GL_{N+1}(\C)\), \(P \in V\), and set:
\[
 A \cdot P
\bydef
P\big((AX)^{(0)}, (AX)^{(1)}, \dotsc).
\]
This action respects the natural filtration and grading on \(V\) defined in the Introduction. Furthermore, a straightforward but crucial observation is that the set \(V^{\Diff}\) of differentially homogeneous polynomials forms a sub-representation of \(V\). As a matter fact, these two facts are so important that we record them in a lemma:
\begin{lemma}
\label{lemma: elementary}
The following two facts hold:
\begin{itemize}
\item{} The natural action of the general linear group on \(V\) preserves the filtration and the grading;
\item{} The vector space \(V^{\Diff}\) forms a sub-representation of \(V\).
\end{itemize}
\end{lemma}

The goal of this Section \ref{sect: Schmidt-Kolchin} is to study the structure of differentially homogeneous polynomials, and eventually prove the Schmidt--Kolchin conjecture.

\subsection{Families of differentially homogeneous polynomials.}
\label{subs: family}
\subsubsection{Construction of differentially homogeneous polynomials}
\label{subs: construction}
Fix \(d\geq 1\) a natural number. For any \(d\)-uple \(\mathbi{n}=(n_{1}, \dotsc, n_{d}) \in \Set{0, \dotsc, N}^{d}\) and any family of polynomials in one variable \(\mathbi{R}\bydef (R_{1}, \dotsc, R_{d}) \in \C[t]^{\oplus d}\), one can define a differentially homogeneous polynomial (denoted \(\Wronsk(R_{1}X_{n_{1}}, \dotsc, R_{d}X_{n_{d}})\)) of degree \(d\) as follows. Consider the following element in \(\mathcal{M}_{1,d}(V[t])\)
\[
L_{\mathbi{n}}(\mathbi{R})\bydef(R_{1}(t)X_{n_{1}}, \dotsc, R_{d}(t)X_{n_{d}}).
\]
For a natural number \(r\geq 0\), denote by \(L_{\mathbi{n}}(\mathbi{R})^{(r)}\) the line obtained by replacing, for \(1 \leq k \leq d\), the \(k\)th entry of \(L_{\mathbi{n}}(\mathbi{R})\) (i.e.  \(R_{k}(t)X_{n_{k}}\)) by
\[
\sum\limits_{j=0}^{r} \binom{r}{j} R_{k}(t)^{(r-j)} X_{n_{k}}^{(j)}.
\]
Form then the \(d\times d\) square matrix
\[
W_{\mathbi{n}}(\mathbi{R})
\bydef
\begin{pmatrix}
L_{\mathbi{n}}(\mathbi{R})
\\
L_{\mathbi{n}}^{(1)}(\mathbi{R})
\\
\cdot
\\
\cdot 
\\
L_{\mathbi{n}}^{(d-1)}(\mathbi{R})
\end{pmatrix},
\]
and consider its determinant. One has the following proposition:
\begin{proposition}
\label{prop: def Wronsk}
The determinant \(\det(W_{\mathbi{n}}(\mathbi{R}))_{\vert t=0}\) is a differentially homogeneous polynomial of degree \(d\), denoted by  \(\Wronsk(R_{1}X_{n_{1}}, \dotsc, R_{d}X_{n_{d}})\).
\end{proposition}
\begin{proof}
Let \(Q \in \C[T]\). From the definition of the action of \(Q\) on \(V\) (defined in the Introduction) and from the very definition of \(L_{\mathbi{n}}^{(r)}(\mathbi{R})\), observe that the following equality holds
\[
Q \cdot L_{\mathbi{n}}^{(r)}(\mathbi{R})
=
(QL_{\mathbi{n}}(\mathbi{R}))^{(r)}.
\]
%Indeed, the action of \(Q\) on the \(k\)th entry of \(L_{\mathbi{i}}^{(r)}\) is by definition as follows
%\begin{eqnarray*}
%Q \cdot \Big(\sum\limits_{j=0}^{r} \binom{r}{j} (t^{k})^{(r-j)} X_{i_{k}}^{(j)}\Big)
%&
%=
%&
%\sum\limits_{j=0}^{r} \binom{r}{j} (t^{k})^{(r-j)} (QX_{i_{k}})^{(j)}
%\\
%&
%=
%&
%\sum\limits_{j=0}^{r} \binom{r}{j} Q^{(r-j)} (t^{k}X_{i_{k}})^{(j)}.
%\end{eqnarray*}
%Here, the second equality follows formally from the following fact: for three complex functions \(f, g, h\) depending on one variable, the following equality holds (by the Leibnitz rule, and the commutativity of the ground field \(\C\))
 %\[
%(fgh)^{(r)}
%=
%\sum\limits_{j=0}^{r} \binom{r}{j} f^{(r-j)}(gh)^{(j)}
%=
%\sum\limits_{j=0}^{r} \binom{r}{j} g^{(r-j)}(fh)^{(j)}.
%\]
Therefore, one deduces that the polynomial
 \[
 Q \cdot \det(W_{\mathbi{n}}(\mathbi{R})) \in V[t,T]
 \]
is equal to the determinant of the \(d\times d\) square matrix whose \(r\)th line is equal to \((QL_{\mathbi{n}}(\mathbi{R}))^{(r)}\).
Using elementary operations on the lines as well as the anti-linearity of the determinant, one deduces that 
\[
Q \cdot \det(W_{\mathbi{n}}(\mathbi{R}))
=
Q^{d} \det(W_{\mathbi{n}}(\mathbi{R})).
\]
This shows that each coefficient in \(t\) of \(\det(W_{\mathbi{n}}(\mathbi{R}))\) is indeed differentially homogeneous of degree \(d\), and so does the evaluation at \(t=0\).
\end{proof}

As observed by Reinhart in \cite{Reinhart_1999}, one can extract a free family of \((N+1)^{d}\) differentially homogeneous polynomials as follows\footnote{In his paper, the detailed construction is different from the one presented here, but they amount to the same. I take this opportunity to thank Gleb Pogudin, as he indicated me the construction presented here.}:
\begin{proposition}
\label{prop: family}
Consider the following data \(\mathcal{P}\):
\begin{itemize}
\item{} A \((N+1)\)-uple \(\mathbi{m}=(m_{0}, \dotsc, m_{N}) \in \N^{N+1}\) such that \(\abs{\mathbi{m}}=d\). Let 
\[
i_{1} < \dotsb < i_{r}
\]
be the ordered set of indexes \(i\) such that \(m_{i} \neq 0\).
\item{}
For any \(1 \leq \ell \leq r\), an increasing sequence of \(m_{i_{\ell}}\) integers satisfying:
\[
0 \leq \alpha_{i_{\ell},1} < \dotsb < \alpha_{i_{\ell},m_{i_{\ell}}} < m_{i_{1}}+\dotsb+m_{i_{\ell}}.
\]
\end{itemize}
To any such data \(\mathcal{P}\), denote by \(W_{\mathcal{P}}\) the following differentially homogeneous polynomial of degree \(d\)
\[
W_{\mathcal{P}}
\bydef
\Wronsk(t^{\alpha_{i_{1},1}}X_{i_{1}}, \dotsc, t^{\alpha_{i_{1},m_{i_{1}}}}X_{i_{1}}, \dotsc,  t^{\alpha_{i_{r},1}}X_{i_{r}}, \dotsc, t^{\alpha_{i_{r},m_{i_{r}}}}X_{i_{r}}).
\]
Then, the family \((W_{\mathcal{P}})_{\mathcal{P}}\) forms a free family of \((N+1)^{d}\) differentially homogeneous polynomials of degree \(d\) (and order less or equal than \((d-1)\)).
\end{proposition}
\begin{proof}
First, one easily computes that there are indeed \((N+1)^{d}\) such data \(\mathcal{P}\). Indeed, for a fixed \((N+1)\)-uple \(\mathbi{m}\) as in the statement, there are:
\[
\binom{d}{m_{i_{r}}} \binom{d-m_{i_{r}}}{m_{i_{r-1}}} \dotsb \binom{d-(m_{i_{r}}+\dotsb +m_{i_{2}})}{m_{i_{1}}}
=
\binom{d}{m_{i_{1}}, \dotsc, m_{i_{r}}}
\]
ways of creating sequences of integers satisfying the second item in the statement (one has used the usual notation for multinomial coefficients). The counting of the data \(\mathcal{P}\) follows from Newton multinomial formula.

To see that the constructed differentially homogeneous polynomials of degree \(d\) are independent, proceed as follows. Put the following order on the variables \((X_{i}^{(k)})\):
\[
X_{0}^{(d-1)} > \dotsb > X_{0}^{(0)} > X_{1}^{(d-1)} > \dotsb > X_{N}^{(0)}.
\]
This induces a total order on the monomials by first comparing the term of higher order in \(X_{0}^{(\cdot)}\), then the degree in that factor, and so on. Now, the key observation is that monomials of least order in the polynomials \(W_{\mathcal{P}}\)'s are pairwise distinct, and respectively equal to the product of the entries on the diagonal of the defining matrix. This shows the linear independency, and finishes the proof of the proposition.
\end{proof}
\begin{example}
Take \(N=2\), \(d=6\), and consider \(\Wronsk(X_{0},tX_{0}, tX_{1},t^{3}X_{1}, X_{2}, t^{4}X_{2})\). This differentially homogeneous polynomial is given by the determinant of the following matrix:
\[
\begin{pmatrix}
X_{0}^{(0)} & 0 & 0 & 0 & X_{2}^{(0)} & 0
\\
X_{0}^{(1)} & X_{0}^{(0)} & X_{1}^{(0)} & 0 & X_{2}^{(1)} & 0
\\
X_{0}^{(2)} & 2X_{0}^{(1)} & 2X_{1}^{(1)} & 0 & X_{2}^{(2)} & 0
\\
X_{0}^{(3)} & 3X_{0}^{(2)} & 3X_{1}^{(2)} & 6X_{1}^{(0)} & X_{2}^{(3)} & 0
\\
X_{0}^{(4)} & 4X_{0}^{(3)} & 4X_{1}^{(3)} & 24X_{1}^{(1)} & X_{2}^{(4)} & 24X_{2}^{(0)}
\\
X_{0}^{(5)} & 5X_{0}^{(4)} & 5X_{1}^{(4)} & 60X_{1}^{(2)} & X_{2}^{(5)} & 120X_{2}^{(1)}
\end{pmatrix}.
\]
\end{example}
Denote by 
\[
\overline{V_{d}^{\Diff}} \bydef \Span_{\C}((W_{\mathcal{P}})_{\mathcal{P}})
\]
the \((N+1)^{d}\)-dimensional \(\C\)-vector space spanned by the family given in the previous Proposition \ref{prop: family}. The differential polynomials \((W_{\mathcal{P}})_{\mathcal{P}}\) will be referred to as the \textsl{canonical basis of \(\overline{V_{d}^{\Diff}}\)}.

According to the Schmidt--Kolchin conjecture, the vector space \(\overline{V_{d}^{\Diff}}\) should be the whole vector space of differentially homogeneous polynomials of degree \(d\). In particular, by Lemma \ref{lemma: elementary}, this should be a representation of the general linear group \(\GL_{N+1}(\C)\). Proving that it is indeed the case is the object of the next section.
\begin{remark}
There are other natural families of differentially homogeneous polynomials that one can construct. For instance, one can consider, for \(\theta \in \C^{*}\), the following family:
\[
\Big\{\Wronsk\big(X_{n_{1}},(\theta+t)X_{n_{2}}, \dotsc, (\theta+t)^{d-1}X_{n_{d}})\big) \ | \ \mathbi{n}=(n_{1}, \dotsc, n_{d}) \in \Set{0, \dotsc, N}^{d}\Big\}.
\]
Whereas one immediately sees that the vector space spanned by the previous family defines is a sub-representation of \(V_{d}^{\Diff}\), it is not obvious at all to determine its dimension. We conjecture that, for a generic value of \(\theta\), this is \((N+1)^{d}\)-dimensional, but we were not able to prove it.
\end{remark}

\subsubsection{Invariance under the action of the general linear group.}
\label{subs: invariance}
The sought invariance of \(\overline{V_{d}^{\Diff}}\) under the action of \(\GL_{N+1}(\C)\) will follow from more general considerations that we develop in Appendix \ref{appendix: A}. We chose to put this part in an appendix because, in itself, it has little to do with differentially homogeneous polynomials. It is rather a more general statement about families of determinants of a certain shape. 

For now, let us state a corollary of the main statement in Appendix \ref{appendix: A}, which will be useful for our purposes. Let us fix \(Y_{1}, \dotsc, Y_{d}\) formal variables, along with their formal derivatives \(Y_{1}^{(k)}, \dotsc, Y_{d}^{(k)}\), with \(k\) a natural number. We consider the following family of differentially homogeneous polynomials of degree \(d\):
\[
\mathcal{F}_{d}
\bydef 
\big(\Wronsk(t^{\alpha_{1}}Y_{1}, \dotsc, t^{\alpha_{d}}Y_{d})\big)_{0 \leq \alpha_{i} \leq d-1},
\]
as well as the following subfamily:
\[
\tilde{\mathcal{F}_{d}} 
\bydef
\big(\Wronsk(t^{\alpha_{1}}Y_{1}, \dotsc, t^{\alpha_{d}}Y_{d})\big)_{0 \leq \alpha_{i} \leq i-1}.
\]
\begin{remark}
Note that the canonical basis of \(\overline{V_{d}^{\Diff}}\) defined in the previous Section \ref{subs: construction} is obtained from the family \(\tilde{\mathcal{F}_{d}}\) by suitable substitutions \(Y_{i}=X_{n_{i}}\), \(n_{i} \in \Set{0, \dotsc, N}\).
\end{remark}
\begin{remark}
If one the \(\alpha_{i}\)'s in the previous definition is taken larger or equal than \(d\), then it is easily seen that the resulting differential polynomial is identically zero.
\end{remark}
As a corollary of Proposition \ref{prop: appendix A} and remark \ref{remark: appendix A}, one deduces immediately the following:
\begin{proposition}
\label{prop: main prop stab}
The \(\C\)-vector space spanned by \(\tilde{\mathcal{F}_{d}}\) is the same than the one spanned by \(\mathcal{F}_{d}\).
\end{proposition}
\begin{proof}
See Appendix \ref{appendix: A}.
\end{proof}
This proposition implies in turn the following:
\begin{theorem}
The \(\C\)-vector space \(\overline{V_{d}^{\Diff}}\) is a sub-representation of \(V_{d}^{\Diff}\) (for the natural action of the general linear group \(\GL_{N+1}(\C)\)).
\end{theorem}
\begin{proof}
Consider a data \(\mathcal{P}\) as described in Proposition \ref{prop: family}, and the differential polynomial \(W_{\mathcal{P}}\) associated to it:
\[
W_{\mathcal{P}}
=
\Wronsk(t^{\alpha_{i_{1},1}}X_{i_{1}}, \dotsc, t^{\alpha_{i_{1},m_{i_{1}}}}X_{i_{1}}, \dotsc,  t^{\alpha_{i_{r},1}}X_{i_{r}}, \dotsc, t^{\alpha_{i_{r},m_{i_{r}}}}X_{i_{r}}).
\]
Denoting
\[
\mathbi{\alpha}
=
(\alpha_{i_{1},1}, \dotsc, \alpha_{i_{r}, m_{r}}) \in \N^{d},
\]
observe that, by construction, one has the inequality
 \(
 \alpha_{i} \leq i-1
 \)
for any \(1 \leq i \leq d\). Accordingly, the differentially homogeneous polynomial \(W_{\mathcal{P}}\) is nothing but
\[
\Wronsk(t^{\alpha_{1}}Y_{1}, \dotsc, t^{\alpha_{d}}Y_{d})
\]
after the evaluations 
\[
Y_{1}=\dotsc=Y_{m_{i_{1}}}=X_{i_{1}}, Y_{m_{i_{1}}+1}=\dotsb=Y_{m_{i_{1}}+m_{i_{2}}}=X_{i_{2}}, \dotsc .
\]
Reciprocally, being given
\[
\Wronsk(t^{\alpha_{1}}Y_{1}, \dotsc, t^{\alpha_{d}}Y_{d})
\]
with \(\alpha_{i} \leq i-1\) for any \(1 \leq i \leq d-1\), this yields, \textsl{up to sign}, a differential polynomial in the canonical basis of \(\overline{V_{d}^{\Diff}}\) after evaluations of the form
\[
Y_{1}=\dotsc=Y_{m_{i_{1}}}=X_{i_{1}}, Y_{m_{i_{1}}+1}=\dotsb=Y_{m_{i_{1}}+m_{i_{2}}}=X_{i_{2}}, \dotsc ,
\]
provided that the following holds:
\begin{itemize}
\item{}
\(\alpha_{1}, \dotsc, \alpha_{m_{1}}\) are pairwise distinct;
\item{}
\(\alpha_{m_{1}+1}, \dotsc, \alpha_{m_{1}+m_{2}}\) are pairwise distinct;
\item{}
\textsl{etc}.
\end{itemize}
If these conditions are not satisfied, then the evaluation yields trivially zero by anti-linearity of the determinant.

Now, if one lets an invertible matrix \(A \in \GL_{N+1}(\C)\) act on \(W_{\mathcal{P}}\), then, by expanding out, one finds that \(A \cdot W_{\mathcal{P}}\) writes as a linear combination of terms of the form
\[
W\bydef \Wronsk(t^{\beta_{1}}Y_{1}, \dotsc, t^{\beta_{d}}Y_{d})_{\vert Y_{0}=X_{n_{0}}, \dotsc, Y_{d}=X_{n_{d}}}
\]
for some \(0 \leq \beta_{i}\leq d-1\) and \(0 \leq n_{0} \leq \dotsb \leq n_{d-1} \leq N\). By Proposition \ref{prop: main prop stab}, one can suppose that such a term actually satisfies the inequalities
\[
\beta_{i} \leq i-1
\]
for any \(1 \leq i \leq d-1\). By the above, the differential polynomial \(W\) is either zero, or an element in \(\overline{V_{d}^{\Diff}}\). This proves that \(\overline{V_{d}^{\Diff}}\) is left stable under the action of \(\GL_{N+1}(\C)\). It is therefore a sub-representation of \(V_{d}^{\Diff}\), so that the proof is complete.
\end{proof}

For our purposes, the following corollary will be of particular importance:
\begin{proposition}
\label{prop: estimate under}
For any partition \(\mathbi{\lambda} \vdash d\) with at most \((N+1)\) parts, the following inequality holds:
\[
\dim \big(V^{\Diff}_{\hw(\mathbi{\lambda})}\big)^{(d-1)}
\geq f_{\mathbi{\lambda}},
\]
where one recalls that \(V^{\Diff}_{\hw(\mathbi{\lambda})}\) is the vector space of \(\mathbi{\lambda}\)-highest weight vectors in \(V_{d}^{\Diff}\).
\end{proposition}
\begin{proof}
Now that one knows that \(\overline{V_{d}^{\Diff}}\) is a (finite-dimensional) representation, one can compute its character. Since the canonical basis of \(\overline{V_{d}^{\Diff}}\) is made of weight vectors, one computes immediately that:
\begin{eqnarray*}
\chi_{\overline{V_{d}^{\Diff}}}(x_{0}, \dotsc, x_{N})
&
=
&
\sum\limits_{\substack{\mathbi{m} \in \N^{N+1} \\ \abs{\mathbi{m}}=d}}
\binom{d}{m_{0}, \dotsc, m_{N}}
x_{0}^{m_{0}} \dotsb x_{N}^{m_{N}}
\\
&
=
&
(x_{0}+\dotsb+x_{N})^{d}.
\end{eqnarray*}
Using for instance Theorem \ref{thm: struct gen}, one sees that the symmetric polynomial \((x_{0}+\dotsb+x_{N})^{d}\) writes:
\[
(x_{0}+\dotsb+x_{N})^{d}\
=
\sum\limits_{\mathbi{\lambda} \vdash d} f_{\mathbi{\lambda}}s_{\mathbi{\lambda}}.
\]
Note that, if \(\mathbi{\lambda}\) has more than \((N+1)\) parts, then \(s_{\mathbi{\lambda}}=0\).
Since a representation is uniquely determined by its character, one deduces that \(\overline{V_{d}^{\Diff}}\) has \(f_{\mathbi{\lambda}}\) independent \(\mathbi{\lambda}\)-highest weight vectors. The proposition now follows from the straightforward observation that \(\overline{V_{d}^{\Diff}} \subset (V_{d}^{\Diff})^{(d-1)}\).
\end{proof}

\subsection{A natural basis for the highest weight vectors of \(V_{d}^{(k)}\).}
\label{subs: hw vectors}
Note that there is a tautological isomorphism of representation
\[
V_{d}^{(k)}
\simeq
\bigoplus_{\substack{\mathbi{a} \in \N^{k+1} \\ \abs{\mathbi{a}}=d}}
S^{a_{0}}\C^{N+1} \otimes \dotsb \otimes S^{a_{k}} \C^{N+1}.
\]
Elementary \textsl{plethysm} using the so-called \textsl{Pieri's formula} (see e.g. \cite{Fulton}[I.2.2 \& II.8.3]) tells that \(V_{d}^{(k)}\) decomposes into irreducible representations as follows:
\begin{equation}
\label{eq: sum of rep}
V_{d}^{(k)}
\simeq
\bigoplus_{\mathbi{\lambda} \vdash d}
\bigoplus_{\substack{\mathbi{a} \in \N^{k+1} \\ \abs{\mathbi{a}}=d}}
S^{\mathbi{\lambda}}(\C^{N+1})^{\oplus K_{\mathbi{\lambda}, \mathbi{a}}}.
\end{equation}
An easy corollary of the decomposition \eqref{eq: sum of rep} is the following:
\begin{lemma}
\label{lemma: hw count}
The number of irreducible representations of \(V_{d}^{(k)}\) is equal to
\[
\sum\limits_{\mathbi{\lambda} \vdash d} \dim S^{\mathbi{\lambda}}\C^{k+1}.
\]
\end{lemma}
\begin{proof}
This follows from the equality \eqref{eq: sum of rep}, and the following identity:
\begin{equation*}
\sum\limits_{\substack{\mathbi{a} \in \N^{k+1} \\ \abs{\mathbi{a}}=d}} K_{\mathbi{\lambda}, \mathbi{a}}
=
\dim S^{\mathbi{\lambda}}\C^{k+1}
\end{equation*}
(see the equality \eqref{eq: Kotska} and Proposition \ref{prop: dim Schur}).
\end{proof}
In order to define a basis of highest weight vectors inside \(V_{d}^{(k)}\), let us first introduce some notations.
To any sequence of natural numbers \(\mathbi{i}=(i_{0}, \dotsc, i_{r})\), with \(0 \leq r \leq N\), define
\[
D_{\mathbi{i}} 
\bydef
\det
\begin{pmatrix}
X_{0}^{(i_{0})} & \cdot & \cdot & X_{r}^{(i_{0})}
\\
\cdot & & & \cdot
\\
\cdot & & & \cdot
\\
X_{0}^{(i_{r})} & \cdot & \cdot & X_{r}^{(i_{r})}
\end{pmatrix}.
\]
To any Young tableau \(T\) of shape \(\mathbi{\lambda}=(\lambda_{0} \geq \dotsb \geq \lambda_{s})\) with at most  \((N+1)\) rows, filled with the numbers \(\Set{0, \dotsc, k}\), define
\[
D_{T}
\bydef
D_{T(\cdot,1)} \times \dotsb \times D_{T(\cdot, \lambda_{0})}.
\]
Here, for \(1 \leq i \leq \lambda_{0}\), the symbol \(T(.,i)\) represents the sequence of integers read off (from top to bottom) from the \(i\)th column of the Young tableau \(T\). 

Now, for \(\mathbi{\lambda} \vdash d\) a partition of \(d\), consider the following sub-vector space of \(V_{d}^{(k)}\): 
\[
\mathcal{D}_{\mathbi{\lambda}}^{(k)}
\bydef
\Span_{\C}\big\{D_{T} \ | \ \text{T Young tableau of shape \(\mathbi{\lambda}\), filled with \(\Set{0, \dotsc, k}\)}\big\}
\]
One has the following important lemma:
\begin{proposition}
\label{prop: hw vectors}
The following holds:
\begin{enumerate}
\item{}
Let \(\mathbi{\lambda} \vdash d\) be a partition of the integer \(d\) with at most \((N+1)\) parts.
The family
\[
(D_{T})_{T}
\]
where \(T\) runs over \textsl{semi-standard} Young tableaux of shape \(\mathbi{\lambda}\) filled with \(\Set{0, \dotsc, k}\)
\begin{enumerate}[(i)]
\item{} is a free family of \(\mathbi{\lambda}\)-highest weight vectors in \(V_{d}^{(k)}\);
\item{} spans the vector space \(\mathcal{D}_{\mathbi{\lambda}}^{(k)}\), which is canonically isomorphic to \(S^{\mathbi{\lambda}}\C^{k+1}\).
\end{enumerate}

\item{} 
The direct sum
\[
\bigoplus_{\mathbi{\lambda} \vdash d} \mathcal{D}_{\mathbi{\lambda}}^{(k)}
\]
spans all the highest weight vectors of \(V_{d}^{(k)}\) (note that if \(\mathbi{\lambda}\) has more than \((N+1)\) parts, then \(\mathcal{D}_{\mathbi{\lambda}}^{(k)}\) is zero).
\end{enumerate}
\end{proposition}
\begin{proof}
Let \(\mathbi{\lambda}\) be a partition with at most \((N+1)\) parts, and let \(T\) be a Young tableau of shape \(\mathbi{\lambda}\) filled with the numbers \(\Set{0, \dotsc, k}\). A straightforward application of the anti-linearity of the determinant shows that \(D_{T}\) is indeed a \(\mathbi{\lambda}\)-highest weight vector in \(V_{d}^{(k)}\).
The fact that the family 
\[
\big(D_{T}\big)_{\text{\(T\) semi-standard Young tableaux of shape \(\mathbi{\lambda}\) filled with \(\Set{0, \dotsc, k}\)}}
\]
 is a free family that spans the vector space \(\mathcal{D}_{\mathbi{\lambda}}^{(k)}\) is in particular the content of \cite{Fulton}[II. 8.1 Corollary of Theorem 1]. This very result also shows that \(\mathcal{D}_{\mathbi{\lambda}}^{(k)}\) is canonically isomorphic to \(S^{\mathbi{\lambda}}\C^{k+1}\).
 
The last part of the statement follows immediately from the previous Lemma \ref{lemma: hw count}: one has indeed exhibited as many independent highest weight vectors as there are irreducible representations in \(V_{d}^{(k)}\).
\end{proof}
Now that we have a basis of highest weight vectors in \(V_{d}^{(k)}\), the goal is the following: evaluate how many independent linear combinations of elements in this basis provide differentially homogeneous polynomials. Tackling this question is the object of the next two Sections \ref{subs: pure algebraic} and \ref{subs: relate}.
\subsection{An intermediate algebraic problem.}
\label{subs: pure algebraic}
Let us start this Section \ref{subs: pure algebraic} with a simple observation. A differentially homogeneous polynomial \(P\) of degree \(d\) must satisfy the following property: for any complex number \(\alpha \in \C\), the following equality holds
\begin{equation*}
P\big((\alpha+T)X, \big((\alpha+T)X\big)^{(1)}, \dotsc\big)_{\vert T=0}
=
\alpha^{d} P\big(X, X^{(1)}, \dotsc \big).
\end{equation*}
This equality is equivalent to saying that, if one makes the substitution 
\begin{equation}
\label{eq: transfo}
X^{(i)} \longleftrightarrow \alpha X^{(i)} + iX^{(i-1)}
\end{equation}
for any \(i \in \N\), then the polynomial \(P\) becomes \(\alpha^{d}P\). One is therefore naturally lead to study differential polynomials satisfying such a property.

The goal of this Section \ref{subs: pure algebraic} is to study a purely algebraic problem related to transformations of the type \eqref{eq: transfo}. We will then relate this algebraic problem to our situation of interest in the next Section \ref{subs: relate}.

\subsubsection{Bounding the dimension of the kernel of a family of nilpotent endomorphisms of \((\C^{k+1})^{\otimes d}\).}
\label{subs: algebraic formulation}
Consider the following nilpotent endomorphism of \(\C^{k+1}\):
\[
J
\bydef
\begin{pmatrix}
0 & 1 & 0 & \cdot & \cdot & 0
\\
0 & 0 & 2 & \cdot & \cdot & 0
\\
\cdot & \cdot & \cdot & \cdot & \cdot & \cdot
\\
\cdot & \cdot & \cdot & \cdot & \cdot & \cdot
\\
\cdot & \cdot & \cdot & \cdot & \cdot & k
\\
0 & 0 & \cdot & \cdot & \cdot & 0
\end{pmatrix}.
\]
Fix \(v\in \C^{k+1}\) a vector such that \(J^{k}v \neq 0\). It induces a natural basis of \(E \bydef (\C^{k+1})^{\otimes d}\) given by:
\[
\Big(\underbrace{J^{\alpha_{1}}v \otimes \dotsb \otimes J^{\alpha_{d}}v}_{\bydef J^{\mathbi{\alpha}}v}\Big)_{0 \leq \alpha_{1}, \dotsc, \alpha_{d} \leq k}.
\]
For any \(1 \leq \ell \leq d\), define an endomorphism \(J^{(\ell)}\) of \(E\) induced by \(J\) as follows. Denote by \((\mathbi{e}_{1}, \dotsc, \mathbi{e}_{d})\) the canonical \(\Z\)-basis of \(\Z^{d}\), and set:
\[
J^{(\ell)}(J^{\mathbi{\alpha}}v)
\bydef
\sum\limits_{1 \leq i_{1} \neq \dotsb \neq i_{\ell}\leq d}J^{\mathbi{\alpha}+\mathbi{e}_{i_{1}}+\dotsb+\mathbi{e}_{i_{\ell}}}v.
\]
Equivalently, if \(v_{1}, \dotsc, v_{d}\) are vectors in \(\C^{k+1}\), the endomorphism \(J^{(\ell)}\) is defined as follows on \(v_{1} \otimes \dotsb \otimes v_{d}\):
\[
J^{(\ell)}(v_{1} \otimes \dotsb \otimes v_{d})
\bydef
\sum\limits_{1 \leq i_{1}\neq \dotsb \neq i_{\ell} \leq d}
(\otimes \prod_{i=1}^{d})(J^{\delta_{i, \Set{i_{1}, \dotsc, i_{\ell}}}}v_{i}).
\]
Here, for any set \(I \subset \N\), one sets
\[
\delta_{\cdot,I}\colon \N \to \Set{0,1}
\]
to be the function that is equal to \(1\) if \(i \in I\), and zero otherwise.

The problematic is the following: we are looking for tensors in \(E=(\C^{k+1})^{\otimes d}\) that are in the kernel of the endomorphisms \(J^{(\ell)}\) for any \(1\leq \ell \leq d\). 
This problem can be reformulated as follows. Consider the isomorphism of vector spaces that identifies
\[
J^{\mathbi{\alpha}}v
\longleftrightarrow
\frac{X_{1}^{k-\alpha_{1}}}{(k-\alpha_{1})!} 
\dotsb
\frac{X_{d}^{k-\alpha_{d}}}{(k-\alpha_{d})!} 
\in \C[X_{1}, \dotsc, X_{d}].
\]
The key observation is that, under this identification, the endomorphism \(J^{(\ell)}\) is nothing but the endomorphism induced by the partial differential equation:
\[
\sum\limits_{1\leq i_{1} \neq \dotsb \neq i_{\ell} \leq d} \frac{\partial^{\ell}}{\partial X_{i_{1}} \dotsb \partial X_{i_{\ell}}}.
\]

With this reformulation, it becomes relevant to study solutions in \(\O(\C^{d})\) of the following system of PDE's with constant coefficients:
\[
(\mathcal{S})\colon
\
\
\Big( 
\sum\limits_{1\leq i_{1} \neq \dotsb \neq i_{\ell} \leq d} \frac{\partial^{\ell}}{\partial X_{i_{1}} \dotsb \partial X_{i_{\ell}}}
\Big)_{1 \leq \ell \leq d}.
\]
Note that by classic considerations on symmetric polynomials, the above system is equivalent to the following system of PDE's:
\[
(\mathcal{S'})\colon
\
\
\Big( 
\sum\limits_{i=1}^{d} \frac{\partial^{\ell}}{(\partial X_{i})^{\ell}} 
\Big)_{1 \leq \ell \leq d}.
\]
Following Section \ref{subs: PDE}, we are lead to study the ideal in \(\C[X_{1}, \dotsc, X_{d}]\)
\[
I\bydef(S_{1}, \dotsc, S_{d}),
\]
where, for \(r\in \N_{\geq1}\), \(S_{r}\) is the \(r\)th Newton polynomial in \(d\) variables, namely:
\[
S_{r}
\bydef 
\sum\limits_{i=1}^{d} X_{i}^{r}.
\]
One has then the following lemma:
\begin{lemma}
\label{lemma: ideal}
The affine variety \(V(I)\) consists of the origin
\[
V(I)=\Set{0},
\]
and one has the equality:
\[
\length\Big(\frac{\C[X_{1}, \dotsc, X_{d}]}{I}\Big)=d!.
\]
\end{lemma}
\begin{proof}
The fact that the affine variety \(V(I)\) consists only of the origin is classic. If \((a_{1}, \dotsc, a_{d}) \in V(I)\), then one has the equality in \(\C[Z]\)
\[
Z^{d}=\prod\limits_{i=1}^{d}(Z-a_{i}),
\]
so that necessarily \(a_{1}=\dotsb=a_{d}=0\).

Let \(m_{0}=(X_{1}, \dotsc, X_{d})\) be the maximal ideal at the origin, and consider the Artinian ring \(M\bydef\frac{\C[X_{1}, \dotsc, X_{d}]}{I}\). One easily sees that
\[
\length(M)
=
\length(M_{m_{0}}),
\]
where \(M_{m_{0}}\) is the localized module at \(m_{0}\). Consider now the following trick. Add one variable \(T\) to the polynomial ring \(\C[X_{1}, \dotsc, X_{d}]\), and denote by
\[
\tilde{I}=(S_{1}, \dotsc, S_{d}) \subset \C[X_{1}, \dotsc, X_{d}, T]
\]
the homogeneous ideal in \(\C[X_{1}, \dotsc, X_{d}, T]\) induced by \(S_{1}, \dotsc, S_{d} \in \C[X_{1}, \dotsc, X_{d}]\). The homogeneous variety \(\Proj(\tilde{M})\), where \(\tilde{M}\bydef \frac{\C[X_{1}, \dotsc, X_{d}, T]}{\tilde{I}}\), consists then of the single point
\[
\infty \bydef [0:\dotsc:0:1] \in \P^{d}.
\]
Denoting \(m_{\infty}\) the ideal sheaf of the closed point \(\infty\), elementary intersection theory (see e.g. \cite{FultonIntersection}[Proposition 8.4, and discussion below]) allows to show that
\begin{eqnarray*}
\length(M)
=
\length(M_{m_{0}})
=
\length(\tilde{M}_{m_{\infty}})
&
=
&
\O_{\P^{d}}(1) \cdot \O_{\P^{d}}(2) \cdot \dotsb \cdot \O_{\P^{d}}(d)
\\
&
=
&
d!.
\end{eqnarray*}
This finishes the proof of the lemma.
\end{proof}
Now, Theorem \ref{thm: PDE} readily implies the following proposition:
\begin{proposition}
\label{prop: estimate1}
One has the following estimate:
\[
\dim\big(
\bigcap_{1 \leq \ell \leq d} \Ker J^{(\ell)} 
\big)
\leq 
d!.
\]
\end{proposition}
\begin{proof}
On the one hand, one has seen that elements in \(\bigcap_{1\leq \ell \leq d} \Ker J^{(\ell)}\) embeds as polynomial solutions of the system of PDE's \((\mathcal{S})\). On the other hand, Theorem \ref{thm: PDE} combined with Lemma \ref{lemma: ideal} shows that the system \((\mathcal{S})\) admits exactly \(d!\) independent solutions in \(\O_{\C^{d}}\supset \C[X_{1}, \dotsc, X_{d}]\). This proves the proposition.
\end{proof}
\begin{remark}
Note that the upper bound depends only on \(d\), and not on \(k\).
\end{remark}

\subsubsection{Equivariance of the endomorphisms \(J^{(\ell)}\) with respect to the action of the symmetric group \(\Sigma_{d}\).}
\label{subs: equiv}
We keep the notations introduced in the previous Section \ref{subs: algebraic formulation}. Recall that there is a natural \textsl{right} action of symmetric group \(\Sigma_{d}\) on \(E=(\C^{k+1})^{\otimes d}\) obtained by permuting the factors:
\[
(v_{1} \otimes \dotsb \otimes v_{d})\cdot \sigma
\bydef
v_{\sigma(1)}\otimes \dotsb \otimes v_{\sigma(d)}.
\]
It makes \(E\) into a right \(\C[\Sigma_{d}]\)-module.
A key observation is that the endomorphisms \((J^{(\ell)})_{1 \leq \ell \leq d}\) commutes with the action of \(\Sigma_{d}\):
\begin{lemma}
\label{lemma: sigma linear}
For any \(\sigma \in \Sigma_{d}\), any \(1 \leq \ell \leq d\) and any \(v_{1} \otimes \dotsb \otimes v_{d}\), the following equality holds:
\[
\big(J^{(\ell)}(v_{1} \otimes \dotsb \otimes v_{d})\big) \cdot \sigma
=
J^{(\ell)}\big((v_{1} \otimes \dotsb \otimes v_{d})\cdot \sigma\big)
\]
\end{lemma}
\begin{proof}
On the one hand, compute that:
\begin{eqnarray*}
\big(J^{(\ell)}(v_{1} \otimes \dotsb \otimes v_{d})\big) \cdot \sigma
&
=
&
\big(
\sum\limits_{1 \leq i_{1}\neq \dotsb \neq i_{\ell} \leq d}
(\otimes \prod_{i=1}^{d})(J^{\delta_{i, \Set{i_{1}, \dotsc, i_{\ell}}}}v_{i})
\big) \cdot \sigma
\\
&
=
&
\sum\limits_{1 \leq i_{1}\neq \dotsb \neq i_{\ell} \leq d}
(\otimes \prod_{i=1}^{d})(J^{\delta_{\sigma(i), \Set{i_{1}, \dotsc, i_{\ell}}}}v_{\sigma(i)}).
\end{eqnarray*}
On the other hand, compute that:
\begin{eqnarray*}
J^{(\ell)}\big((v_{1} \otimes \dotsb \otimes v_{d})\cdot \sigma\big)
&
=
&
\sum\limits_{1 \leq i_{1}\neq \dotsb \neq i_{\ell} \leq d} 
(\otimes \prod_{i=1}^{d})(J^{\delta_{i, \Set{i_{1}, \dotsc, i_{\ell}}}}v_{\sigma(i)})
\\
&
=
&
\sum\limits_{1 \leq i_{1}\neq \dotsb \neq i_{\ell} \leq d} 
(\otimes \prod_{i=1}^{d})(J^{\delta_{i, \Set{\sigma^{-1}(i_{1}), \dotsc, \sigma^{-1}(i_{\ell})}}}v_{\sigma(i)})
\\
&
=
&
\sum\limits_{1 \leq i_{1}\neq \dotsb \neq i_{\ell} \leq d} 
(\otimes \prod_{i=1}^{d})(J^{\delta_{\sigma(i), \Set{i_{1}, \dotsc, i_{\ell}}}}v_{\sigma(i)}).
\end{eqnarray*}
This proves the lemma.
\end{proof}

By the above Lemma \ref{lemma: sigma linear}, one deduces that for any \(1 \leq \ell \leq d\), and any standard Young tableau \(T\) with \(d\) boxes, the endomorphisms \(J^{(\ell)}\) commute with the almost-projections \(E(p_{T})\) (see Section \ref{subs: linear} for notations). In particular, these endomorphisms stabilize \(\Image(E(p_{T}))\). One has then the following elementary result:
\begin{lemma}
\label{lemma: indep}
The dimension of
\[
\bigcap_{1 \leq \ell \leq d} \Ker J^{(\ell)}_{\vert \Image(E(p_{T}))}
\]
is independent of the standard Young tableau of shape \(\mathbi{\lambda} \vdash d\).
\end{lemma}
\begin{proof}
Let \(T\) and \(T'\) be two standard tableaux of shape \(\mathbi{\lambda} \vdash d\). There exists a unique permutation \(\sigma \in \Sigma_{d}\) such that
\[
\sigma \cdot T
=
T'.
\]
Denote by \(E(\sigma)\) the isomorphism
\[
E(\sigma)\colon
\left(
\begin{array}{ccc}
E(\C[\Sigma_{d}]) & \longrightarrow  & E(\C[\Sigma_{d}])
 \\
 v & \longmapsto  & v\otimes_{\C[\Sigma_{d}]} \sigma \end{array}
\right).
\]
Observe that the following equality holds in \(\C[\Sigma_{d}]\):
\[
c_{T'}
=
\sigma \times c_{T} \times \sigma^{-1}.
\]
This implies in turn the following equality
\[
E(p_{T'})=E(\sigma^{-1}) \circ E(p_{T}) \circ E(\sigma).
\] 
In particular, if \(v \in \Image E(p_{T})\cap \Ker J^{(\ell)}\) for some \(1 \leq \ell \leq d\),  then
\[
E(\sigma^{-1})(v) \in \Image E(p_{T'}) \cap \Ker J^{(\ell)}.
\]
This allows to show that \(E(\sigma^{-1})\) realizes a bijection between 
\(\bigcap_{1 \leq \ell \leq d} \Ker J^{(\ell)}_{\vert \Image(E(p_{T}))}\)
and
\(\bigcap_{1 \leq \ell \leq d} \Ker J^{(\ell)}_{\vert \Image(E(p_{T'}))}\). This finishes the proof of the lemma.
\end{proof}

For a partition \(\mathbi{\lambda} \vdash d\), denote by \(T_{\can(\mathbi{\lambda})}\) the \textsl{canonical} standard Young tableau of shape \(\mathbi{\lambda}\), which is defined as follows:
\begin{itemize}
\item{} the first row of the diagram is filled with \(1, 2, \dotsc, \lambda_{1}\);
\item{} the second row of the diagram is filled with \(\lambda_{1}+1, \dotsc, \lambda_{1}+\lambda_{2}\);
\item{} \(\dotsc\)
\end{itemize}
From Proposition \ref{prop: estimate1} and Lemma \ref{lemma: indep}, one deduces the following:
\begin{proposition}
\label{prop: estimate2}
The following inequality holds:
\[
\sum\limits_{\mathbi{\lambda} \vdash d}
f_{\mathbi{\lambda}}
\times
\dim\big(\bigcap_{1 \leq \ell \leq d} \Ker J^{(\ell)}_{\vert \Image(E(p_{T_{\can(\mathbi{\lambda})}}))}\big)
\leq
d!.
\]
\end{proposition}
\begin{proof}
By Theorem \ref{thm: struct gen}, one has the direct sum decomposition
\[
E
=
\bigoplus_{T} \Image(E(p_{T})),
\]
where \(T\) runs over standard tableaux with \(d\) boxes. The statement now follows immediately from Proposition \ref{prop: estimate1} and Lemma \ref{lemma: indep}.
\end{proof}

\subsection{Where one relates the previous problem to the question at hand.}
\label{subs: relate}
Let us now relate the study carried over in the previous Section \ref{subs: pure algebraic} to the Schmidt--Kolchin conjecture. We fix \(\mathbi{\lambda} \vdash d\) a partition of \(d\) \textsl{with at most \((N+1)\) parts}. Following Section \ref{subs: hw vectors} and the very beginning Section \ref{subs: pure algebraic}, we wish to understand polynomials \(P\) in \(\mathcal{D}_{\mathbi{\lambda}}^{(k)}\) (see Lemma \ref{prop: hw vectors} for notations) which satisfy the following identity
\begin{equation}
\label{eq: functional eq}
P(\alpha X, \alpha X^{(1)} + X^{(0)}, \alpha X^{(2)}+2X^{(1)}, \dotsc)
=
\alpha^{d} P(X, X^{(1)}, X^{(2)}, \dotsc)
\end{equation}
for any \(\alpha \in \C\). 
This can be restated as follows.
Denote by
\[
F\bydef \C \cdot X \oplus \C \cdot X^{(1)} \oplus \dotsb \oplus \C \cdot X^{(k)}
\simeq \C^{k+1}.
\]
Observe that the linear group \(\GL(F)\simeq \GL_{k+1}(\C)\) acts naturally on \(\mathcal{D}_{\mathbi{\lambda}}^{(k)}\) by change of variable\footnote{Note that this action has nothing to do with the usual linear action on \(V\).}, i.e. for \(P \in \mathcal{D}_{\mathbi{\lambda}}^{(k)}\) and \(A \in \GL(F)\):
\[
A \cdot P
\bydef 
P(AX, AX^{(1)}, \dotsc, AX^{(k)}).
\]
The equality \eqref{eq: functional eq} is then equivalent to the following equality (see Section \ref{subs: pure algebraic} for the definition of the matrix \(J\)):
\begin{equation}
\label{eq: functional eqbis}
(\alpha \Id + J)\cdot P
=
\alpha^{d} P.
\end{equation}
Note that for an arbitrary \(P \in \mathcal{D}_{\mathbi{\lambda}}^{(k)}\), the polynomial 
\[
(\alpha \Id + J)\cdot P
\]
is a polynomial of degree \(d\) in \(\alpha\), with leading term \(P(X, X^{(1)}, X^{(2)}, \dotsc)\). Therefore, in order to obtain \eqref{eq: functional eq}, we must require the vanishing of all the coefficients in front of the terms \(\alpha^{i}\), \(i<d\). 

For any Young tableau \(T\) with \(d\) boxes, denote
\[
e_{T}
\bydef
(\otimes \prod_{(i,j) \in T})X^{(T(i,j))} \in E \bydef F^{\otimes d},
\]
where the tensor product is taken over the elements in the tableau \(T\), read in the usual fashion, i.e. from left to right and top to bottom.
Let \(T_{\can(\mathbi{\lambda})}\) be the canonical standard tableau with shape \(\mathbi{\lambda}\). For sake of notations, let us denote \(c_{\mathbi{\lambda}}\bydef c_{T_{\can(\mathbi{\lambda})}}\).
We have the following crucial proposition:
\begin{proposition}
\label{prop: isomorphism}
The linear map
\[
e\colon
\left(
\begin{array}{ccc}
 \mathcal{D}_{\mathbi{\lambda}}^{(k)} & \longrightarrow  &  \Image(E(p_{T_{\can(\mathbi{\lambda})}})) \subset E=F^{\otimes d} \\
  D_{T} & \longmapsto  & e_{T}\otimes_{\C[\Sigma_{d}]} c_{\mathbi{\lambda}}
\end{array}
\right)
\]
is well-defined, and is an isomorphism\footnote{The assumption that the partition has at most \((N+1)\) parts is important: this map would be the zero-map otherwise.}.
\end{proposition}
\begin{proof}
It essentially follows from \cite{Fulton}[II.8.1 Lemma 3 \& Theorem 1] and \cite{Fulton}[II.7.4 Proposition 4]. Details are provided in Appendix \ref{appendix: B}.
\end{proof}
\begin{remark}
This proposition is essential, because it allows to almost completely get rid of the dependency on \(N\): the sole dependency on \(N\) for the space on the right lies on the constraint on the partition \(\mathbi{\lambda}\) (it must not have more than \((N+1)\) parts).
\end{remark}
Consider the following natural surjective linear map
\[
\pi\colon
\left(
\begin{array}{ccc}
 E & \longrightarrow  & \mathcal{D}_{\mathbi{\lambda}}^{(k)}   
 \\
  e_{T} & \longmapsto &  D_{T}
     \end{array}
\right),
\]
and observe that one has the following commutative diagram:
\begin{equation}
\label{eq: diagram}
\xymatrix{
E
\ar[rr]^-{E(p_{T_{\can(\mathbi{\lambda})}})}
\ar[rd]^-{\pi}
&
&
\Image(E(p_{T_{\can(\mathbi{\lambda})}}))
\\
&
 \mathcal{D}_{\mathbi{\lambda}}^{(k)}
 \ar[ru]^-{e}
 }.
\end{equation}
A simple but important observation is that the action of \(\GL(F)\) on \(\mathcal{D}_{\mathbi{\lambda}}^{(k)}\) commutes with the projection \(\pi\):
\begin{lemma}
\label{lemma: final inter}
For any \(v \in E\) and any \(A \in \GL(F)\), the following equality holds:
\[
\pi(A\cdot v)
=
A \cdot \pi(v).
\]
\end{lemma}
\begin{proof}
This follows from \cite{Fulton}[II.8.1 Exercices 3 \& 4].
\end{proof}
As an immediate corollary, the action of \(\GL(F)\) commutes with the isomorphism \(e\):
\begin{lemma}
\label{lemma: final1}
For any \(P \in \mathcal{D}_{\mathbi{\lambda}}^{(k)} \) and any \(A \in \GL(F)\), the following equality holds:
\[
e(A\cdot P)
=
A \cdot e(P).
\]
\end{lemma}
\begin{proof}
By linearity, it suffices to prove the equality for \(P=D_{T}\), where \(T\) is a Young tableau of shape \(\mathbi{\lambda}\) filled with \(\Set{0, \dotsc, k}\). By Lemma \ref{lemma: final inter}, compute that
\[
e(A\cdot D_{T})
=
e(A\cdot \pi(e_{T}))
=
e(\pi(A \cdot e_{T})).
\]
By commutativity of the diagram \eqref{eq: diagram}, one has:
\[
e(\pi(A \cdot e_{T}))
=
p_{T_{\can}(\mathbi{\lambda})}(A \cdot e_{T})
=
A \cdot p_{T_{\can}(\mathbi{\lambda})}(e_{T})
=
A \cdot e(D_{T}).
\]
This shows the result.
\end{proof}
Now, the key observation to relate our problem to what we did in Section \ref{subs: pure algebraic} is the following:
\begin{lemma}
\label{lemma: final2}
For any \(v \in E\), and any \(\alpha \in \C\), the following equality holds:
\[
(\alpha \Id + J)
\cdot
v
=
J^{(d)}(v) + \alpha J^{(d-1)}(v) + \dotsb + \alpha^{d-1}J^{(1)}(v) + \alpha^{d}v.
\]
\end{lemma}
\begin{proof}
It suffices to check the equality for \(v=e_{T}\), where \(T\) is a Young tableau of shape \(\mathbi{\lambda}\) filled with \(\Set{0, \dotsc, k}\). This is then a straightforward computation to show the sought equality.
\end{proof}
Denote 
\[
\overset{\sim}{V_{\hw(\mathbi{\lambda})}^{(k)}}
\bydef
\big\{P \in \mathcal{D}_{\mathbi{\lambda}}^{(k)}
\ | \ 
(\alpha \Id + J)\cdot P=\alpha^{d}P
\
\forall \alpha \in \C
\big\}.
\]
As a simple corollary of the previous lemmas, we obtain the following important proposition: 
\begin{proposition}
\label{prop: iso}
One has the following isomorphism of vector spaces:
\[
\overset{\sim}{V_{\hw(\mathbi{\lambda})}^{(k)}}
\simeq
\bigcap_{1 \leq \ell \leq d} \Ker J^{(\ell)}_{\vert \Image(E(p_{T_{\can(\mathbi{\lambda})}}))}.
\]
\end{proposition}
\begin{proof}
Since \(e\) is an equivariant isomorphism (by Proposition \ref{prop: isomorphism} and Lemma \ref{lemma: final1}), there is an isomorphism
\[
\overset{\sim}{V_{\hw(\mathbi{\lambda})}^{(k)}}
\simeq
\Set{v \in \Image(E(p_{T_{\can(\mathbi{\lambda})}})) \ | \  (\alpha \Id + J)\cdot v=\alpha^{d}v}.
\]
Since the endomorphisms \(J^{(\ell)}\) commute with \(E(p_{T_{\can(\mathbi{\lambda})}})\) (by Lemma \ref{lemma: sigma linear}), Lemma \ref{lemma: final2} implies that the set on the right is nothing but
\[
\bigcap_{1 \leq \ell \leq d} \Ker J^{(\ell)}_{\vert \Image(E(p_{T_{\can(\mathbi{\lambda})}}))}.
\]
This finishes the proof of the proposition.
\end{proof}

We would like to emphasize again that, in the above statement, the dependency on \(N\) lies only in the partition \(\mathbi{\lambda} \vdash d\) (namely, it must not have more than \((N+1)\) parts).
We now have all the tools to finish the proof of the Schmidt--Kolchin conjecture

\subsection{Proof of the Schmidt--Kolchin conjecture.}
\label{subs: proof}
 Fix \(d \in \N_{\geq 1}\) a natural number, and fix \(k \geq d-1\).
 \begin{remark}
One chooses to take \(k \geq d-1\) simply because the canonical basis of \(\overline{V_{d}^{\Diff}}\) lies in \((V_{d}^{\Diff})^{(d-1)}\).
\end{remark}
For the moment, suppose that the natural number \(N\) is greater or equal than  \(d-1\).
One knows by the beginning of Section \ref{subs: pure algebraic} that for any partition \(\mathbi{\lambda} \vdash d\) with at most \((N+1)\) parts, the following inclusion holds:
\begin{equation}
\label{eq: inclusion}
V_{\hw(\mathbi{\lambda})}^{(k)}
\subset
\overset{\sim}{V_{\hw(\mathbi{\lambda})}^{(k)}}
=
\big\{P \in \mathcal{D}_{\mathbi{\lambda}}^{(k)}
\ | \ 
(\alpha \Id + J)\cdot P=\alpha^{d}P
\
\forall \alpha \in \C
\big\}.
\end{equation}
Note that, since \(N \geq d-1\), any partition \(\mathbi{\lambda} \vdash d\) has at most \(N+1\) parts.
From Proposition \ref{prop: estimate under}, one therefore deduces the inequality
\begin{equation*}
\dim \overset{\sim}{V_{\hw(\mathbi{\lambda})}^{(k)}}
\geq 
f_{\mathbi{\lambda}}
\end{equation*}
for any partition \(\mathbi{\lambda} \vdash d\).
By Proposition \ref{prop: iso}, this is the same as the following inequality:
\begin{equation}
\label{eq: ineq final}
\dim\big(\bigcap_{1 \leq \ell \leq d} \Ker J^{(\ell)}_{\vert \Image(E(p_{T_{\can(\mathbi{\lambda})}}))}\big)
\geq 
f_{\mathbi{\lambda}}.
\end{equation}
Now, using Proposition \ref{prop: estimate2} and the above inequality \eqref{eq: ineq final}, one obtains the following string of inequalities:
\[
\sum\limits_{\mathbi{\lambda} \vdash d} f_{\mathbi{\lambda}}^{2}
\leq 
\sum\limits_{\mathbi{\lambda} \vdash d}
f_{\mathbi{\lambda}}
\times
\dim\big(\bigcap_{1 \leq \ell \leq d} \Ker J^{(\ell)}_{\vert \Image(E(p_{T_{\can(\mathbi{\lambda})}}))}\big)
\leq
d!.
\]
By Proposition \ref{prop: RSK}, one has the equality
\[
\sum\limits_{\mathbi{\lambda} \vdash d}
f_{\mathbi{\lambda}}^{2}
=
d!.
\]
Therefore, one deduces that, for any \(\mathbi{\lambda} \vdash d\), the following equality holds:
\begin{equation}
\label{eq: eq dim}
\dim\big(\bigcap_{1 \leq \ell \leq d} \Ker J^{(\ell)}_{\vert \Image(E(p_{T_{\can(\mathbi{\lambda})}}))}\big)
=
f_{\mathbi{\lambda}}.
\end{equation}
\begin{remark}
Note that one went in the opposite direction than in Section \ref{subs: relate}. Namely, one has used informations on differentially homogeneous polynomials (obtained in Section \ref{subs: family}) to deduce properties on the algebraic problem of Section \ref{subs: pure algebraic}. This is precisely why Proposition \ref{prop: estimate under} is so crucial.
\end{remark}
Return now to the case where \(N \in \N_{\geq 1}\) is arbitrary. For any partition \(\mathbi{\lambda}\) with at most \((N+1)\) parts, Proposition \ref{prop: estimate under} combined with the inclusion \eqref{eq: inclusion} and the equality \eqref{eq: eq dim} forces the equality
\[
\dim \big(V_{\hw(\mathbi{\lambda})}^{(k)}\big)
=
f_{\mathbi{\lambda}}.
\]
By standard considerations in representation theory recalled in Section \ref{subs: linear}, one deduces the equality
\[
\dim \big((V_{d}^{\Diff})^{(k)}\big)
=
\sum\limits_{\mathbi{\lambda} \vdash d}
f_{\mathbi{\lambda}}
\times
\dim \big(S^{\mathbi{\lambda}} \C^{N+1}\big).
\]
Note that if \(\mathbi{\lambda}\) has more than \((N+1)\) parts, then \(S^{\mathbi{\lambda}} \C^{N+1}=(0)\).
By Propositions \ref{prop: dim Schur} and \ref{prop: RSK}, this implies in turn the equality:
\[
\dim \big((V_{d}^{\Diff})^{(k)}\big)
=
(N+1)^{d}.
\]
As this holds for any \(k \geq (d-1)\), this shows that
\(
V_{d}^{\Diff}
=
(V_{d}^{\Diff})^{(d-1)},
\)
and that
\[
\dim(V_{d}^{\Diff})
=
(N+1)^{d}.
\]
This finishes the proof of the Schmidt--Kolchin conjecture:
\begin{theorem}
\label{thm: mainthm1 corpus}
The following equality holds:
\[
V_{d}^{\Diff}=\overline{V_{d}^{\Diff}}=\Span_{\C}(W_{\mathcal{P}})_{\mathcal{P}}.
\]
In particular, \(V_{d}^{\Diff}=(V_{d}^{\Diff})^{(d-1)}\) and \(\dim V_{d}^{\Diff}=(N+1)^{d}\).
\end{theorem}

\section{Differentially homogeneous polynomials and twisted jet differentials on projective spaces.}
\label{sect: diff jets}
\subsection{Green--Griffiths vector bundles on complex manifolds.}
\label{subs: GG}
The reference for this Section \ref{subs: GG} is \cite{Santa}. The goal is to quickly recall the definition of the so-called \textsl{Green--Griffiths vector bundles}, whose global sections are called \textsl{jet differentials}.

Let \(X\) be a complex manifold of dimension \(N\). For \(k \in \N_{\geq 1}\), define the bundle \(J_{k}X\) of \(k\)-jets of \(1\)-germs of holomorphic maps \(\gamma\colon (\C,0) \rightarrow X\) on the complex manifold \(X\) as follows. Consider an atlas \((U_{i}, \varphi_{i})_{i \in I}\) of \(X\), and for \(i \in I\), consider the (trivial) bundle on \(U_{i}\) whose fiber over \(x \in U_{i}\) is the \(\C\)-vector space of dimension \(N\times k\)
\[
\left\{
\Big(\frac{\diff^{i}}{\diff z^{i}}(\varphi_{i} \circ \gamma)(0)\Big)_{1 \leq i \leq k}
\
\big|
\
\gamma: (\C,0) \rightarrow (X,x) \ \text{holomorphic \(1\)-germ} \
\right\}.
\]
Glue these trivial bundles 
\(U_{i} \times \C^{N\times k}\) via (the maps naturally induced by) the transition maps \(\varphi_{j} \circ \varphi_{i}^{-1}\) to obtain (up to isomorphism) the bundle \(J_{k}X\). The general formula to change charts involves higher order derivatives of the transition maps as soon as \(k>1\), and in particular, it does not preserve the structure of vector space of the fibers: \(J_{k}X\) is not a vector bundle for \(k>1\). For sake of notation, denote
\(
\diff_{k}\gamma
\)
the element in \((J_{k}X)_{x}\) defined by the holomorphic \(1\)-germ \(\gamma\colon (\C,0) \to (X,x)\).
\begin{example}
In the case where \(k=1\), \(J_{1}X\) is nothing but the (holomorphic) tangent bundle \(TX\).
\end{example}
 
The torus \(\C^{*}\) acts on the fibers of \(J_{k}X\) as follows
\[
 \lambda
\cdot 
\Big(\frac{\diff^{i}}{\diff z^{i}}(\varphi_{i} \circ \gamma)(0)\Big)_{1 \leq i \leq k}
=
\Big(\lambda^{i} \frac{\diff^{i}}{\diff z^{i}}(\varphi_{i} \circ \gamma)(0)\Big)_{1 \leq i \leq k},
\]
 where \(\lambda \in \C^{*}\). It is indeed straightforward to see that this action commutes with a  change of chart, and is thus well defined on the bundle \(J_{k}X\). More generally, the group of biholomorphisms of \( (\C, 0) \) acts on \(J_{k}X\) by setting
\[
\psi \cdot \diff_{k} \gamma
\bydef
\diff_{k}(\gamma \circ \psi),
\]
where \(\psi\) is a biholomorphism of \((\C,0)\) and \(\gamma: (\C, 0) \to X\) a holomorphic \(1\)-germ.

Define the vector bundle \(E_{k,n}X\) of jet differentials (of \(1\)-germs) of order \(k \geq 1\) and degree \(n \geq 1\) as follows. Construct the bundle whose fiber over \(x \in X\) is the vector space of complex valued polynomials \(Q\) of degree \(n\) on the fiber \((J_{k}X)_{x}\), i.e. 
for any \(\lambda \in \C\) and any \(1\)-germ \(\gamma\colon (\C,0) \to (X,x) \), the polynomial \(Q\) satisfies the equality
\[
Q(\lambda \cdot \diff_{k}\gamma)
=
\lambda^{n} Q(\diff_{k}\gamma).
\]
The formula to compute multi-derivatives of compositions of maps allows to see that the structure of vector space of the fibers is preserved under a change of chart. The bundle \(E_{k,n}X\) is thus a vector bundle, and is usually called a \textsl{Green--Griffiths vector bundle (of order \(k\) and degree \(n\))}. A global section
\[
P \in H^{0}(X, E_{k,n}X)
\]
is usually called a \textsl{jet differential (of order \(k\) and degree \(n\))}. 
If the complex manifold is projective, and polarized by an ample line bundle \(\O_{X}(1)\), one is naturally lead to consider twisted vector bundles of the form \(E_{k,n}X(d)\bydef E_{k, n}X\otimes L^{d}\), with \(d \in \Z\). A global section of such a twisted Green--Griffiths bundle is called accordingly a \textsl{twisted jet differential}.

For reasons that will become transparent in the next Section \ref{subs: link}, we will rather consider the following direct sums of vector bundles for \(k \in \N_{\geq 0}\):
\[
E_{k, \infty}X
\bydef 
\bigoplus_{n=0}^{\infty} E_{k,n}X.
\]
Here, by convention, one has set
\[
E_{k,0}X\bydef \O_{X}
\]
for any \(k \in \N_{\geq 0}\). For sake of consistency, a global section of \(L^{d}\), with \(d \in \N\), will be considered to be a twisted jet differential of order \(0\).

\subsection{Differentially homogeneous polynomials and twisted jet differentials on projective spaces.}
\label{subs: link}
The link between differentially homogeneous polynomials and twisted global sections of Green--Griffiths bundles of projective spaces is given by the following proposition:
\begin{proposition}
\label{prop: link}
For any \(k \geq 1\) and any \(d \in \N\), there is a natural isomorphism of vector spaces
\[
H^{0}(\P^{N}, E_{k,\infty}\P^{N} (d))
\simeq
(V^{\Diff}_{d})^{(k)}.
\]
\end{proposition}
\begin{proof}
Let \(W \in (V^{\Diff}_{d})^{(k)}\). On each standard affine open set \(U_{i}\bydef \Set{X_{i}\neq 0}\), define
\[
W_{i}
\bydef 
W(\overset{\wedge_{i}}{X^{(0)}}, \dotsc, \overset{\wedge_{i}}{X^{(k)}}).
\]
Here, one has set for \(0\leq \ell \leq k\)
\[
\overset{\wedge_{i}}{X^{(\ell)}}
\bydef
\Big(\big(\frac{X_{0}}{X_{i}}\big)^{(\ell)}, \dotsc, \big(\frac{X_{N}}{X_{i}}\big)^{(\ell)}\Big),
\]
where the upper index in parenthesis stands for the usual formula to differentiate, e.g.:
\[
\big(\frac{X_{0}}{X_{i}}\big)^{(1)}
=
\frac{X_{0}^{(1)}X_{i}-X_{0}X_{i}^{(1)}}{X_{i}^{2}}.
\]
On the corresponding trivialization of \(E_{k,\infty}\P^{N}_{\vert U_{i}}\), the differential polynomial \(W_{i}\) (in the variables \((\frac{X_{\ell}}{X_{i}})_{0\leq \ell \leq N}\)) induces in a natural fashion a section of \(E_{k,\infty}\P^{N}_{\vert U_{i}}\). Now, the fact that \(W\) is differentially homogeneous of degree \(d\) implies that the following equality holds:
\[
X_{i}^{d}W_{i}
=
X_{j}^{d}W_{j}.
\]
In particular, this shows that the local sections \(W_{i}\) glue to a global section \(\mathcal{W}\) of \(E_{k,\infty}\P^{N}(d)\).

Reciprocally, let \(\mathcal{W} \in H^{0}(\P^{N}, E_{k,\infty}\P^{N}(d))\) be a global section of \(E_{k,\infty}\P^{N}(d)\). On each trivializing open set \(U_{i}\), the restricted section \(W_{i}\bydef \mathcal{W}_{\vert U_{i}}\) defines a differential polynomial in the variables \(\overset{\wedge_{i}}{X}=(\frac{X_{\ell}}{X_{i}})_{0\leq \ell \leq N}\). Consider
\[
W\bydef X_{i}^{d}W_{i},
\]
and note that the rational function (rational in the variables \(X^{(0)}\), polynomial in the variables \(X^{(1)}, X^{(2)}, \dotsc\)) obtained is independent of \(0 \leq i \leq N\). In particular, this shows that \(W\) is in fact a differential polynomial. Let then \(Q \in \C[T]\), and compute that
\begin{eqnarray*}
W((QX)^{(0)}, \dotsc, (QX)^{(k)})
&
=
&
Q(T)^{d}X_{i}^{d}W_{i}\big(\overset{\wedge_{i}}{X^{(0)}}, \dotsc, \overset{\wedge_{i}}{X^{(k)}}\big)
\\
&
=
&
Q(T)^{d}W.
\end{eqnarray*}
Therefore, \(W\) is differentially homogeneous of degree \(d\).

Both maps are clearly inverse to each other, hence the proposition.
\end{proof}
As a straightforward corollary of the previous proposition, one has the following important result:
\begin{theorem}
There is a one-to-one correspondance between twisted jet differential on \(\P^{N}\), and differentially homogeneous polynomials in the variables \(X=(X_{0}, \dotsc, X_{N})\).
\end{theorem}
\begin{proof}
In view of the previous Proposition \ref{prop: link}, it suffices to show that for any \(d<0\), and any \(k, n \in \N\), there is no global section of 
\[
E_{k,n}\P^{N}(d).
\]
This is standard: one uses the natural filtration one can put on \(E_{k,n}\P^{N}(d)\)(see e.g. \cite{Santa}[§6]), and one deduces the result from classic vanishing results on the cohomology groups of the graded components (which are well-understood).
\end{proof}
So far, we have not used yet what we know about the structure of differentially homogeneous polynomials. Using the previous correspondance, and the results of Section \ref{sect: Schmidt-Kolchin}, we can deduce interesting results on the space of global sections of twisted Green--Griffiths bundles of projective spaces: this is the object of the next Section \ref{subs: H^0 GG}.

\subsection{On the zeroth cohomology group of twisted Green--Griffiths bundles of projective spaces.}
\label{subs: H^0 GG}
In view of Theorem \ref{thm: mainthm1 corpus}, Proposition \ref{prop: family} and Proposition \ref{prop: link}, all we have to do is establish exactly what global sections the elements in the canonical basis of \(\overline{V_{d}^{\Diff}}=V_{d}^{\Diff}\) define. 

There is a notion of \textsl{weight} for differential polynomials that we have not defined yet. It corresponds to the notion of \textsl{degree} in the setting of Green--Griffiths bundles of \(\P^{N}\). It is defined as follows: the \textsl{weight} of a monomial \(m\bydef \prod_{i=0}^{N} \prod_{k \in \N} (X_{i}^{(k)})^{n_{i,k}}\) is defined as
\[
w(m)
\bydef
\sum\limits_{i=0}^{N}\sum\limits_{k \geq 1} kn_{i,k}.
\]
The definition extends in the usual fashion to the whole vector space \(V\) of differential polynomials. An elements in \(V\) whose monomials share the same weight \(w\) is called \textsl{\(w\)-isobaric}.

Consider now a data \(\mathcal{P}\) as described in Proposition \ref{prop: family}, namely:
\begin{itemize}
\item{} A \((N+1)\)-uple \(\mathbi{m}=(m_{0}, \dotsc, m_{N}) \in \N^{N+1}\) such that \(\abs{\mathbi{m}}=d\). Let 
\[
i_{1} < \dotsb < i_{r}
\]
be the ordered set of indexes \(i\) such that \(m_{i} \neq 0\).
\item{}
For any \(1 \leq \ell \leq r\), an increasing sequence of \(m_{i_{\ell}}\) integers satisfying:
\[
0 \leq \alpha_{i_{\ell},1} < \dotsb < \alpha_{i_{\ell},m_{i_{\ell}}} < m_{i_{1}}+\dotsb+m_{i_{\ell}}.
\]
\end{itemize}
Denote
\[
\mathbi{\alpha}
\bydef
(\alpha_{i_{1},1}, \dotsc, \alpha_{i_{r},m_{r}}) \in \N^{d},
\]
and observe (once again) that \(\alpha_{i} \leq i-1\) for any \(1 \leq i \leq d\). We have the following elementary proposition (see also \cite{Reinhart_1999}[Corollary 2.5] for the case \(N=1\)):
\begin{proposition}
\label{prop: isobar}
The differentially homogeneous polynomial \(W_{\mathcal{P}}\) is of order exactly
 \[
 k=\max \Set{d-1-\alpha_{i} \ | \ 1 \leq i \leq d},
 \]
and is isobaric of weight
\[
w(W_{\mathcal{P}})
=
\frac{d(d-1)}{2}-\abs{\mathbi{\alpha}}.
\]
Furthermore, the weight of \(W_{\mathcal{P}}\) is bounded above by
\[
\lfloor(1-\frac{1}{N+1})\frac{d^{2}}{2}\rfloor
\]
\end{proposition}
\begin{proof}
Consider \(\sigma \in \Sigma_{d}\) a permutation. In the development of the determinant defining \(W_{\mathcal{P}}\), the monomial associated to the permutation \(\sigma\) is the following
\[
m_{\sigma}
=
\prod\limits_{a=1}^{r}
\prod\limits_{b=1}^{m_{a}}
X_{i_{a}}^{(\sigma(m_{1}+\dotsb+m_{a-1}+b)-\alpha_{i_{a},b}-1)}.
\]
Here, by convention, the monomial is zero if one of the upper indexes in parenthesis is negative. One immediately deduces the first two statements of the lemma.

As for the last statement, distinguish two cases, depending on whether or not \(d \geq N+1\). On the one hand, if \(d \leq N+1\), then \(\abs{\mathbi{\alpha}}\) is minimal for the following choice:
\begin{itemize}
\item{} \(\mathbi{m}=(\underbrace{1,\dotsc,1}_{\times d},0,\dotsc,0)\);
\item{} \(\alpha_{0,1}=\alpha_{1,1}=\dotsb=\alpha_{d-1,1}=0\).
\end{itemize}
On the other hand, if \(d>N+1\), then \(\abs{\mathbi{\alpha}}\) is minimal for the following choice:
\begin{itemize}
\item{} \(\mathbi{m}=(r, \dotsc,r,\underbrace{r+1,\dotsc,r+1}_{\times c})\), with \(r=\lfloor \frac{d}{N+1} \rfloor\) and \(c=d-r(N+1)\);
\item{} \(\alpha_{\ell,1}=0, \dotsc, \alpha_{\ell,r}=r-1\) for \(0 \leq \ell \leq N+1-c\);
\item{}  \(\alpha_{\ell,1}=0, \dotsc, \alpha_{\ell,r+1}=r\) for \(N+2-c\leq \ell \leq N+1\).
\end{itemize}
In both cases, one checks that the given upper bound is valid, finishing the proof.
\end{proof}

As a corollary of this Lemma \ref{prop: isobar}, along with Proposition \ref{prop: link} and Theorem \ref{thm: mainthm1 corpus} , we deduce the second main result of this paper:
\begin{theorem}
\label{thm: mainthm2 corpus}
Let \(d \in \Z\), \(k \geq 0\) and \(n \geq 0\) be natural numbers. Then:
\begin{enumerate}
\item{} for any \(k \geq d-1\) and any \(n \in \N\), one has the isomorphism
 \[
 H^{0}(\P^{N}, E_{k,n}\P^{N}(d)) \simeq H^{0}(\P^{N}, E_{d-1,n}\P^{N}(d));
 \]
 \item{} the following equality holds:
\[
\sum\limits_{n=0}^{\infty} \dim H^{0}(\P^{N},E_{d-1,n}\P^{N}(d)) = (N+1)^{d};
\]
\item{} for any \(k \in \N\) and any \(n > \lfloor(1-\frac{1}{N+1})\frac{d^{2}}{2}\rfloor\), one has the vanishing: 
\[
H^{0}(\P^{N}, E_{k,n}\P^{N}(d))=(0).
\]
\end{enumerate}
\end{theorem}
\begin{remark}
Of course, one can say much more than this theorem, since we have explicitly every global sections of \(E_{k,n}\P^{N}(d)\) for all \(k,n,d\). Combinatorially speaking, it is however not simple to completely classify the canonical basis of \(V_{d}^{\Diff}\) according to orders and weights.
\end{remark}
\begin{proof}
The first two items follow immediately from Proposition \ref{prop: link} and Theorem \ref{thm: mainthm1 corpus}. As for the last item, it follows from the above Lemma \ref{prop: isobar}.
\end{proof}
\newpage

\appendix

\section{}
\label{appendix: A}

Let us fix \(d \in \N_{\geq 1}\) a natural number, and \(N \in \End(\C^{d})\) a nilpotent endomorphism of nilpotency index \(d\). For \(0 \leq \alpha_{1}, \dotsc, \alpha_{d} \leq d-1\) a family of natural numbers, define the following polynomial:
\[
P_{\alpha_{1}, \dotsc, \alpha_{d}}
\colon
\left(
\begin{array}{ccc}
(\C^{d})^{\oplus d} 
& 
\longrightarrow  
&  
\C
 \\
 (v_{1}, \dotsc, v_{d}) 
 & 
 \longmapsto  
 & 
 \det(N^{\alpha_{1}}v_{1}, \dotsc, N^{\alpha_{d}}v_{d})
\end{array}
\right).\footnote{We have fixed the canonical basis as a reference, i.e. the vectors \(N^{\alpha_{i}}v_{i}\) are written in the canonical basis, and one takes the determinant of the matrix formed by concatenation of these columns.}
\]
Consider the following family of polynomials
\[
\mathcal{F}\bydef \big(P_{\alpha_{1}, \dotsc, \alpha_{d}}\big)_{0 \leq \alpha_{1}, \dotsc, \alpha_{d} \leq d-1}
\]
as well as the subfamily
\[
\tilde{\mathcal{F}}\bydef \big(P_{\alpha_{1}, \dotsc, \alpha_{d}}\big)_{0 \leq \alpha_{i} \leq i-1}.
\]
\begin{remark}
If one of the \(\alpha_{i}\)'s is taken to be greater or equal than \(d\), then the determinant is trivially zero (since \(N^{d}=0\)).
\end{remark}
\begin{remark}
\label{remark: appendix A}
The (differential) polynomials defined in Section \ref{subs: invariance} are obtained from the previous family as follows:
\begin{itemize}
\item{} replace formally the variables \(v_{i}=(v_{i,1}, \dotsc, v_{i,d})\) by \((Y_{i}^{(0)}, \dotsc, Y_{i}^{(d-1)})\);
\item{} consider the nilpotent matrix \(N\) defined on the canonical basis \((e_{1}, \dotsc, e_{d})\) of \(\C^{d}\) as follows:
\[
\left\{
\begin{array}{ll}
Ne_{i}=i\times e_{i+1} \ \  \text{for \(1 \leq i \leq d-1\)} 
\\
Ne_{d}=0
\end{array}
\right..
\]
\end{itemize}
Note that another choice of nilpotent matrix \(N\) would not yield a differentially homogeneous polynomial: this choice is very specific so as to obtain the differential homogeneity.
\end{remark}
The goal of this Appendix \ref{appendix: A} is to show the following:
\[
\Span_{\C}(\mathcal{F})
=
\Span_{\C}(\tilde{\mathcal{F}}).
\]
This will follow from the following general identity:
\begin{lemma}
\label{lemma: general identity}
For any \(0 \leq i \leq d\), and for any vector \(v_{1}, \dotsc, v_{d-i+1} \in \C^{d}\), the following equality holds:
\[
\sum\limits_{\substack{\mathbi{\alpha} \in \N^{d-i+1} \\ \abs{\mathbi{\alpha}}=i}} 
N^{\alpha_{1}}v_{1} \wedge \dotsb \wedge N^{\alpha_{d-i+1}}v_{d-i+1}
= 0.
\]
\end{lemma}
\begin{proof}
One actually proves, by induction on \(d \geq 1\), the following stronger statement: for any \(0 \leq i \leq d\), for any vector \(v_{1}, \dotsc, v_{d-i+1} \in \C^{d}\), and for any \(j \geq i\), the following equality holds:
\[
\sum\limits_{\substack{\mathbi{\alpha} \in \N^{d-i+1} \\ \abs{\mathbi{\alpha}}=j}} 
N^{\alpha_{1}}v_{1} \wedge \dotsb \wedge N^{\alpha_{d-i+1}}v_{d-i+1}
= 0.
\]
If \(d=1\), the result is obvious since \(N=0\). Suppose therefore that \(d \geq 2\), and let \(0 \leq i \leq d\). If \(i=d\), the statement is obvious since \(N^{d}=0\). If \(i=0\), the statement is also obvious since the wedge product of \((d+1)\) vectors lying in a vector space of dimension \(d\) is always zero. Suppose therefore that \(1 \leq i < d\), and let \(j \geq i\).

Note that, by multi-linearity, it suffices to check the sought vanishing when the vectors \(v_{1}, \dotsc, v_{d-i+1}\) run over a fixed basis of \(\C^{d}\). Choose then \(e_{1}, \dotsc, e_{d}\) a basis adapted to \(N\), in the sense that \(Ne_{k}=e_{k+1}\) for \(1 \leq k \leq d\). Here, by convention, one sets \(e_{k}\bydef 0\) whenever \(k \not\in \Set{1, \dotsc, d}\).

Consider then \(e_{k_{1}}, \dotsc, e_{k_{d-i+1}}\), where \(\Set{k_{1}, \dotsc, k_{d-i+1}} \subset \Set{1, \dotsc, d}\). Up to a sign, for any permutation \(\sigma\) of \(\Set{1, \dotsc, d-i+1}\), one has the equality:
\[
\sum\limits_{\substack{\mathbi{\alpha} \in \N^{d-i+1} \\ \abs{\mathbi{\alpha}}=j}} 
N^{\alpha_{1}}e_{k_{\sigma(1)}} \wedge \dotsb \wedge N^{\alpha_{d-i+1}}e_{k_{\sigma(d-i+1)}}
= 
\pm
\sum\limits_{\substack{\mathbi{\alpha} \in \N^{d-i+1} \\ \abs{\mathbi{\alpha}}=j}} 
N^{\alpha_{1}}e_{k_{1}} \wedge \dotsb \wedge N^{\alpha_{d-i+1}}e_{k_{d-i+1}}.
\]
Therefore, one can always suppose that \(k_{1}  \leq \dotsb \leq k_{d-i+1}\).

Denote \(V_{2}\bydef \Vect(e_{2}, \dotsc, e_{d})\), and observe that \(N\) induces a nilpotent endormorphism of \(V_{2}\), with nilpotency index equal to \(d-1\). Discuss according to the values of \(k_{1}\) and \(k_{2}\). Suppose first that \(k_{1} \geq 2\). In this case, the sought vanishing follows from the induction hypothesis applied to \((V_{2}, N_{\vert V_{2}})\) for \(i'=i-1\) and \(j'=j \geq i'\). Then, suppose that \(k_{1}=1\) and \(k_{2} \geq 2\). Write:
\begin{eqnarray*}
&
&
\sum\limits_{\substack{\mathbi{\alpha} \in \N^{d-i+1} \\ \abs{\mathbi{\alpha}}=j}} 
N^{\alpha_{1}}e_{k_{1}} \wedge \dotsb \wedge N^{\alpha_{d-i+1}}e_{k_{d-i+1}}
\\
&
=
&
\sum\limits_{\substack{\mathbi{\alpha} \in \N^{d-i+1} \\ \abs{\mathbi{\alpha}}=j}} 
e_{k_{1}+\alpha_{1}} \wedge \dotsb \wedge e_{k_{d-i+1}+\alpha_{d-i+1}}
\\
&
=
&
e_{1}
\wedge
\underbrace{\sum\limits_{\substack{\mathbi{\alpha} \in \N^{d-i} \\ \abs{\mathbi{\alpha}}=j}}
e_{k_{2}+\alpha_{2}} \wedge \dotsb \wedge e_{k_{d-i+1}+\alpha_{d-i+1}}}_{\bydef \Sigma_{1}}
+
\underbrace{\sum\limits_{\substack{\mathbi{\alpha} \in \N^{d-i+1} \\ \abs{\mathbi{\alpha}}=j-1}}
e_{2+\alpha_{1}} \wedge e_{k_{2}+\alpha_{2}} \wedge \dotsb \wedge e_{k_{d-i+1}+\alpha_{d-i+1}}}_{\bydef \Sigma_{2}}.
\end{eqnarray*}
On the one hand, the induction hypothesis applied to \((V_{2}, N_{\vert V_{2}})\) for \(i'=i-1\) and \(j'=j-1 \geq i'\) implies that \(\Sigma_{2}=0\). On the other hand, the induction hypothesis applied to  \((V_{2}, N_{\vert V_{2}})\) for \(i'=i-1\) and \(j'=j \geq i'\) implies that \(\Sigma_{1}=0\). Accordingly, the sought vanishing follows. Suppose finally that \(k_{1}=k_{2}=1\). Decompose:
\begin{eqnarray*}
I 
&
\bydef
&
\Set{\mathbi{\alpha} \in \N^{d-i+1} \ | \ \abs{\mathbi{\alpha}}=j}
\\
&
=
&
\underbrace{\Set{\mathbi{\alpha} \in I \ | \ \alpha_{1}> \alpha_{2}}}_{\bydef I^{+}}
\sqcup
\underbrace{\Set{\mathbi{\alpha} \in I \ | \ \alpha_{1}=\alpha_{2}}}_{\bydef I^{0}}
\sqcup
\underbrace{\Set{\mathbi{\alpha} \in I \ | \ \alpha_{1}<\alpha_{2}}}_{\bydef I^{-}}.
\end{eqnarray*}
One then computes that:
\begin{eqnarray*}
&
&
\sum\limits_{\substack{\mathbi{\alpha} \in \N^{d-i+1} \\ \abs{\mathbi{\alpha}}=j}} 
e_{k_{1}+\alpha_{1}} \wedge \dotsb \wedge e_{k_{d-i+1}+\alpha_{d-i+1}}
\\
&
=
&
\underbrace{\sum\limits_{\substack{\mathbi{\alpha} \in I^{+}}}
e_{k_{1}+\alpha_{1}} \wedge \dotsb \wedge e_{k_{d-i+1}+\alpha_{d-i+1}}}_{\Sigma_{1}}
+
\underbrace{\sum\limits_{\substack{\mathbi{\alpha} \in I^{-}}}
e_{k_{1}+\alpha_{1}} \wedge \dotsb \wedge e_{k_{d-i+1}+\alpha_{d-i+1}}}_{\Sigma_{2}}.
\end{eqnarray*}
Now, observe that the permutation of the first two coordinates in \(\N^{d-i+1}\) induces a bijection between \(I_{+}\) and \(I_{-}\). This implies in turn that, pairwise, terms in \(\Sigma_{1}\) and \(\Sigma_{2}\) cancel each other. This finishes the induction, and the proof of the lemma.

\end{proof}
As a corollary, the following identities hold:
\begin{corollary}
\label{cor: identities}
Let \(1 \leq i \leq d\), and let \(\beta_{1}, \dotsc, \beta_{i-1}\) be natural numbers such that \(0 \leq \beta_{k} \leq k-1\) for any \(1 \leq k \leq i-1\). Then the following equality holds:
\[
\sum\limits_{\substack{(\alpha_{i}, \dotsc, \alpha_{d}) \in \N^{d-i+1} \\ \alpha_{i} + \dotsb + \alpha_{d}=i}}
P_{\beta_{1}, \dotsc, \beta_{i-1},\alpha_{i}, \dotsc, \alpha_{d}}
=
0.
\]
\end{corollary}
\begin{proof}
Observe first that it suffices to prove the equality for \(\beta_{1}=\dotsb=\beta_{i-1}=0\). Indeed, it suffices to notice that, for any \(\alpha_{i}, \dotsc, \alpha_{d}\), the following equality holds (by very definition):
\[
P_{\beta_{1}, \dotsc, \beta_{i-1}, \alpha_{i}, \dotsc, \alpha_{d}}(v_{1}, \dotsc, v_{d})
=
P_{0, \dotsc, 0, \alpha_{i}, \dotsc, \alpha_{d}}(N^{\beta_{1}}v_{1}, \dotsc, N^{\beta_{i-1}}v_{i-1},v_{i},\dotsc, v_{d}).
\]
Then, observe that the sought equality is equivalent to requiring that, for any \(v_{1}, \dotsc, v_{d} \in \C^{d}\), the following equality holds:
\[
\sum\limits_{\substack{(\alpha_{i}, \dotsc, \alpha_{d}) \in \N^{d-i+1} \\ \alpha_{i} + \dotsb + \alpha_{d}=i}}
\det(v_{1}, \dotsc, v_{i-1}, N^{\alpha_{i}}v_{i}, \dotsc, N^{\alpha_{d}}v_{d})
=
0.
\]
The result now follows immediately from the previous Lemma \ref{lemma: general identity}
\end{proof}
We can now prove the main result of this Appendix \ref{appendix: A}:
\begin{proposition}
\label{prop: appendix A}
The vector space spanned by the family \(\mathcal{F}\) is the same as the one spanned by the family \(\tilde{\mathcal{F}}\).
\end{proposition}
\begin{proof}
Let \(W=P_{\alpha_{1}, \dotsc, \alpha_{d}} \in \mathcal{F}\), with \(0 \leq \alpha_{i} \leq d-1\). One must show that \(W\) can be written as a linear combination of elements in the family \(\tilde{\mathcal{F}}\). Proceed as follows. If \(\alpha_{1}=0\), move onto the next step. Otherwise, successive applications of the formula of Corollary \ref{cor: identities} for \(i=1\) allows to write \(W\) as a linear combination of elements in \(\mathcal{F}\) of the form
\[
W'\bydef P_{0, \beta_{2}, \dotsc, \beta_{d}},
\]
with \(0 \leq \beta_{i} \leq d-1\). If \(0 \leq \beta_{2} \leq 1\), move onto the next step. Otherwise, successive applications of the formula of Corollary \ref{cor: identities} for \(i=2\) allows to write \(W'\) as a linear combination of elements in \(\mathcal{F}\) of the form
\[
W''\bydef P_{0, \gamma_{2}, \gamma_{3}, \dotsc, \gamma_{d}},
\]
with \(0 \leq \gamma_{2} \leq 1\). Continuing in this fashion, one indeed shows that \(W\) can be written as a linear combination of elements in \(\tilde{\mathcal{F}}\), proving the proposition.

\end{proof}

\section{}
\label{appendix: B}
The goal of this Appendix \ref{appendix: B} is to detail the proof of Proposition \ref{prop: isomorphism}:
\begin{proposition}
\label{prop: appendix B}
The linear map
\[
e\colon
\left(
\begin{array}{ccc}
 \mathcal{D}_{\mathbi{\lambda}}^{(k)} & \longrightarrow  &  \Image(E(p_{T_{\can(\mathbi{\lambda})}})) \subset E
 \\
  D_{T} & \longmapsto  & e_{T}\otimes_{\C[\Sigma_{d}]} c_{\mathbi{\lambda}}
\end{array}
\right)
\]
is well-defined, and is an isomorphism. 
\end{proposition}
\begin{proof}
From \cite{Fulton}[II.8.1 Lemma 3 \& Theorem 1], the following two facts hold. First, the polynomials \(D_{T}\)'s, where \(T\) runs over Young tableaux of shape \(\mathbi{\lambda}\) filled with \(\Set{0, \dotsc, k}\), satisfy the so-called \textsl{exchange relations}, namely:
\begin{enumerate}[(i)]
\item{} \(D_{T}+D_{T'}=0\) if \(T'\) is obtained from \(T\) by interchanging two entries in a column;
\item{} \(D_{T}-\sum\limits D_{S}=0\), where the sum is over all Young tableaux \(S\) obtained from \(T\) by an exchange operation of the following type:
\begin{itemize}
\item{} fix \(2\) consecutive columns in \(T\), say the \((i-1)\)th and the \(i\)th column;
\item{} fix a number \(\ell\) that is less than the size of \(i\)th column, and choose \(\ell\) boxes in the \((i-1)\)th column;
\item{} exchange the \(\ell\) entries of the boxes in the \((i-1)\)th column with the first \(\ell\) entries of the \(i\)th column, while maintaining the vertical order in each. 
\end{itemize}
\end{enumerate}
Second, as indicated in Lemma \ref{prop: hw vectors}, the polynomials \(D_{T}\)'s, where \(T\) runs over semi-standard tableaux filled with \(\Set{0, \dotsc, k}\), form a basis \(\mathcal{B}\) of \(\mathcal{D}_{\mathbi{\lambda}}^{(k)}\). Namely, they form a free family, and for any non-semi-standard tableau \(T\), one can use relations of the above type \((i)\) and \((ii)\) to express (uniquely) \(D_{T}\) as a linear combinations of elements in \(\mathcal{B}\).

Consider now one of the two exchange operations as described in \((i)\) or \((ii)\). Note that such an operation is completely described by a suitable permutation \(\sigma \in \Sigma_{d}\), once one has fixed a labelling of the boxes of the Young frame of shape \(\mathbi{\lambda}\): here, one takes the canonical one, namely the Young tableau \(T_{\can(\mathbi{\lambda})}\).
Fix \(T\) a Young tableau of shape \(\mathbi{\lambda}\) filled with \(\Set{0, \dotsc, k}\), and observe that the following equality holds:
\[
e(D_{\sigma \cdot T})
=
e_{T}\otimes_{\C[\Sigma_{d}]} \sigma \times c_{\mathbi{\lambda}}.
\]

Therefore, proving that the map is well-defined amounts to proving that the Young symmetrizer \(c_{\mathbi{\lambda}}\) satisfies the following two properties:
\begin{center}
\begin{enumerate}[a)]
\item{} \(c_{\mathbi{\lambda}}+\sigma \times c_{\mathbi{\lambda}}=0\) for any permutation \(\sigma\) corresponding to an exchange of type \((i)\);
\item{} \(c_{\mathbi{\lambda}}-\sum\limits_{\sigma} \sigma \times c_{\mathbi{\lambda}}=0\) where the sum runs over permutations corresponding to exchanges as described in \((ii)\).
\end{enumerate}
\end{center}

These properties are indeed satisfied, and this can be justified as follows. A \textsl{tabloïd} is defined to be an equivalence class of Young tableaux, two tableaux being equivalent if their corresponding rows contain the same entries. Denote by \(M^{\mathbi{\lambda}}\) the complex vector space spanned by tabloïds of shape \(\mathbi{\lambda}\) filled with \(\Set{1, \dotsc, d}\). It contains an irreducible representation of \(\Sigma_{d}\), called a \textsl{Specht module}, which is defined as follows\footnote{Obviously, it coïncides up to isomorphism with the irreducible representation of \(\Sigma_{d}\) associated to the partition \(\mathbi{\lambda}\).}
\[
S^{\mathbi{\lambda}}
\bydef
\C[\Sigma_{d}] \cdot v_{\mathbi{\lambda}},
\]
where \(v_{\mathbi{\lambda}}\bydef v_{T_{\can(\mathbi{\lambda})}} \bydef \sum\limits_{q \in C(T_{\can(\mathbi{\lambda})})} \epsilon(q) \{q.T_{\can(\mathbi{\lambda})}\}\).
Note that there is a natural isomorphism of left \(\C[\Sigma_{d}]\)-module
\[
\psi\colon
\left(
\begin{array}{ccc}
 \C[\Sigma_{d}]\cdot a_{T_{\can(\mathbi{\lambda})}} & \longrightarrow  & M^{\mathbi{\lambda}}  \\
  \sigma \times a_{T_{\can(\mathbi{\lambda})}} & \longmapsto &  \{\sigma \cdot T_{\can(\mathbi{\lambda})}\}    \end{array}
\right).
\]
Note also that the following equality holds for any \(\sigma \in \Sigma_{d}\)
\[
\psi(\sigma\times c_{\mathbi{\lambda}})
=
\abs{R(T_{\can(\mathbi{\lambda})})}
\times
\sigma \cdot v_{\mathbi{\lambda}}
=
\abs{R(T_{\can(\mathbi{\lambda})})}
\times
v_{\sigma \cdot T_{\can(\mathbi{\lambda})}}.
\]
Here, one has used the identity
\(a_{T_{\can(\mathbi{\lambda})}}^{2}=\abs{R(T_{\can(\mathbi{\lambda})})} \times a_{T_{\can(\mathbi{\lambda})}}\) for the first egality, whereas the second egality follows (almost) by definition. In particular, this shows that the map \(\psi\) sends the irreducible representation \(\C[\Sigma_{d}] \cdot c_{\mathbi{\lambda}} \subset \C[\Sigma_{d}]\) to the irreducible representation \(S^{\mathbi{\lambda}}=\C[\Sigma_{d}] \cdot v_{\mathbi{\lambda}} \subset M^{\mathbi{\lambda}}\).

Now, the sought equalities a) and b) follow from the corresponding ones in \(M^{\mathbi{\lambda}}\). Those equalities are exactly the content of \cite{Fulton}[II.7.4 Corollary of Proposition 4]. This proves that the map \(e\) is well-defined. To see that \(e\) is an isomorphism, observe first that, by its very definition, it is surjective. Then, observe that the target and source space both have the same dimension, since \(E(S^{\mathbi{\lambda}})\simeq S^{\mathbi{\lambda}}\C^{k+1}\) (see Theorem \ref{thm: struct gen} and Proposition \ref{prop: hw vectors}). This finishes the proof.
\end{proof}

\bibliographystyle{alpha}
\bibliography{Schmidt-Kolchin}

\begin{thebibliography}{Dem97}

\bibitem[Boe70]{Boerner}
Hermann Boerner.
\newblock {\em Representations of groups. {W}ith special consideration for the
  needs of modern physics}.
\newblock North-Holland Publishing Co., Amsterdam-London; American Elsevier
  Publishing Co., Inc., New York, english edition, 1970.
\newblock Translated from the German by P. G. Murphy in cooperation with J.
  Mayer-Kalkschmidt and P. Carr.

\bibitem[Dem97]{Santa}
Jean-Pierre Demailly.
\newblock Algebraic criteria for {K}obayashi hyperbolic projective varieties
  and jet differentials.
\newblock In {\em Algebraic geometry---{S}anta {C}ruz 1995}, volume~62 of {\em
  Proc. Sympos. Pure Math.}, pages 285--360. Amer. Math. Soc., Providence, RI,
  1997.

\bibitem[Ehr70]{Ehrenpreis}
Leon Ehrenpreis.
\newblock {\em Fourier analysis in several complex variables}.
\newblock Pure and Applied Mathematics, Vol. XVII. Wiley-Interscience [A
  division of John Wiley \& Sons, Inc.], New York-London-Sydney, 1970.

\bibitem[Eis95]{Eisenbud}
David Eisenbud.
\newblock {\em Commutative algebra}, volume 150 of {\em Graduate Texts in
  Mathematics}.
\newblock Springer-Verlag, New York, 1995.
\newblock With a view toward algebraic geometry.

\bibitem[FH91]{FultonHarris}
William Fulton and Joe Harris.
\newblock {\em Representation theory}, volume 129 of {\em Graduate Texts in
  Mathematics}.
\newblock Springer-Verlag, New York, 1991.
\newblock A first course, Readings in Mathematics.

\bibitem[Ful97]{Fulton}
William Fulton.
\newblock {\em Young tableaux}, volume~35 of {\em London Mathematical Society
  Student Texts}.
\newblock Cambridge University Press, Cambridge, 1997.
\newblock With applications to representation theory and geometry.

\bibitem[Ful98]{FultonIntersection}
William Fulton.
\newblock {\em Intersection theory}, volume~2 of {\em Ergebnisse der Mathematik
  und ihrer Grenzgebiete. 3. Folge. A Series of Modern Surveys in Mathematics
  [Results in Mathematics and Related Areas. 3rd Series. A Series of Modern
  Surveys in Mathematics]}.
\newblock Springer-Verlag, Berlin, second edition, 1998.

\bibitem[Kol92]{Kolchin}
E.~R. Kolchin.
\newblock A problem on differential polynomials.
\newblock In {\em Proceedings of the international conference on algebra
  dedicated to the memory of A. I. Mal'cev, held at Akademgorodok, Novosibirsk,
  USSR, Aug. 21-26, 1989. Part 2}, pages 449--462. Providence, RI: American
  Mathematical Society, 1992.

\bibitem[Obe90]{Oberst90}
Ulrich Oberst.
\newblock Multidimensional constant linear systems.
\newblock {\em Acta Appl. Math.}, 20(1-2):1--175, 1990.

\bibitem[Obe96]{Oberst96}
Ulrich Oberst.
\newblock Finite-dimensional systems of partial differential or difference
  equations.
\newblock {\em Adv. in Appl. Math.}, 17(3):337--356, 1996.

\bibitem[Pal70]{Palamodov}
V.~P. Palamodov.
\newblock {\em Linear differential operators with constant coefficients}.
\newblock Die Grundlehren der mathematischen Wissenschaften, Band 168.
  Springer-Verlag, New York-Berlin, 1970.
\newblock Translated from the Russian by A. A. Brown.

\bibitem[Rei99]{Reinhart_1999}
Georg~M. Reinhart.
\newblock The schmidt{\textendash}kolchin conjecture.
\newblock {\em Journal of Symbolic Computation}, 28(4-5):611--631, oct 1999.

\bibitem[RS96]{Rein-Sit}
Georg~M. Reinhart and William Sit.
\newblock Differentially homogeneous differential polynomials.
\newblock In {\em Proceedings of the 1996 international symposium on symbolic
  and algebraic computation, ISSAC '96, Z\"urich, Switzerland, July 24--26,
  1996}, pages 212--218. New York, NY: ACM Press, 1996.

\bibitem[Sch79]{Schmidt_1979}
Wolfgang~M. Schmidt.
\newblock Contributions to diophantine approximation in fields of series.
\newblock {\em Monatshefte f{\"u}r Mathematik}, 87(2):145--165, jun 1979.

\end{thebibliography}

 \end{document}